\documentclass[11pt]{article}
\pdfoutput=1
\usepackage[top=2cm,bottom=2cm,left=0.5cm,right=4cm]{geometry}
\usepackage{hyperref}
\usepackage{latexsym}
\usepackage{amssymb}
\usepackage{amsmath}
\usepackage{graphicx}
\usepackage{bbm,url}
\usepackage{xcolor}
\usepackage{marginnote}
\usepackage[OT2,OT1]{fontenc}

\begin{document}

\def\Ascr{\mathcal{A}}
\def\Bscr{\mathcal{B}}
\def\Cscr{\mathcal{C}}
\def\Dscr{\mathcal{D}}
\def\Escr{\mathcal{E}}
\def\Fscr{\mathcal{F}}
\def\Hscr{\mathcal{H}}
\def\Iscr{\mathcal{I}}
\def\Jscr{\mathcal{J}}
\def\Mscr{\mathcal{M}}
\def\Nscr{\mathcal{N}}
\def\Pscr{\mathcal{P}}
\def\Qscr{\mathcal{Q}}
\def\Rscr{\mathcal{R}}
\def\Sscr{\mathcal{S}}
\def\Wscr{\mathcal{W}}
\def\Xscr{\mathcal{X}}
\def\cupp{\stackrel{.}{\cup}}
\def\bold{\bf\boldmath}

\newcommand{\rouge}[1]{\textcolor{red}{ #1}}
\newcommand{\bleu}[1]{\textcolor{blue}{ #1}}
\newcommand{\boldheader}[1]{\smallskip\noindent{\bold #1:}\quad}
\newcommand{\PP}{\mbox{\slshape P}}
\newcommand{\NP}{\mbox{\slshape NP}}
\newcommand{\opt}{\mbox{\scriptsize\rm OPT}}
\newcommand{\ec}{\mbox{\scriptsize\rm OPT}_{\small\rm 2EC}}
\newcommand{\lp}{\mbox{\scriptsize\rm LP}}
\newcommand{\inn}{\mbox{\rm in}}
\newcommand{\out}{\mbox{\rm out}}
\newcommand{\dom}{\mbox{\rm dom}}
\newcommand{\ran}{\mbox{\rm ran}}
\newcommand{\MAXSNP}{\mbox{\slshape MAXSNP}}
\newtheorem{theorem}{Theorem} [section]
\newtheorem{lemma}{Lemma}[section]
\newtheorem{corollary}{Corollary}[section]
\newtheorem{proposition}{Proposition}[section]
\newtheorem{fact}{Fact}[section]
\newtheorem{definition}{Definition}[section]
\def\prove{\par \noindent \hbox{{\bf Proof}:}\quad}
\def\endproof{\eol \rightline{$\Box$} \par}
\renewcommand{\endproof}{\hspace*{\fill} {\boldmath $\Box$} \par \vskip0.5em}
\newcommand{\mathendproof}{\vskip-1.8em\hspace*{\fill} {\boldmath $\Box$} \par \vskip1.8em}

\definecolor{orange}{rgb}{1,0.9,0}
\definecolor{violet}{rgb}{0.8,0,1}
\definecolor{darkgreen}{rgb}{0,0.5,0}
\definecolor{grey}{rgb}{0.75,0.75,0.75}

\title {
	\vspace*{-2cm}
		{\huge Matchings: Source, Goal and} \\[3mm]
	{\huge Faithful Companion} \footnote{This article was prepared for the Centenary Issue  of the Journal of the London Mathematical Society, it is accepted for publication in this journal.} \\[3mm]
}
\author{Andr\'as Seb\H{o}\footnote{CNRS, Univ.~Grenoble-Alpes, Laboratoire G-SCOP.}
}
\begingroup
\makeatletter
\let\@fnsymbol\@arabic
\maketitle
\endgroup
\begin{abstract} 
Matchings were among the earliest motivations for graph theory. They subsequently remained a central goal, inspiring the development of new mathematical tools that went well beyond problems directly concerning matchings. 

These tools proved widely applicable, accompanying the growth of graph theory over the past century. A milestone in this trajectory is W.~T.~Tutte, {\em The factorisation of linear graphs}, J. Lond. Math. Soc. (1), {\bf 22}, no. 2, (1947), 107--111., which firmly embedded graph theory—through matchings—into the  body of classical mathematics, in particular,  linear algebra and polynomials. 

We revisit this article of Tutte, presenting its original content, sketching some aspects of its impact until some recent progress, and  trace one of its subsequent lines of development finally leading to a new contribution answering an open challenge and extending known results.

	\smallskip
	\noindent{\footnotesize {\bf keywords}: matchings, factors, path matchings,  Tutte's theorems, the method of variables, randomized polynomial algorithms,  polynomials,  algebraic complexity,  $T$-joins, conservative weightings, bidirected graphs,   jump systems and their intersections}
\end{abstract}

\noindent
{\bf \large Introduction}

\smallskip
After a few isolated eighteenth-century appearances — Euler’s bridges of Königsberg, Jacobi’s work on matrices, contributions of Kempe, Tait,  and Heawood on map colouring and Petersen on matchings in cubic graphs — the ideas that shaped modern graph theory began to crystallize in the twentieth century through the work of Egerváry, Kőnig, Hall, Kuhn, Tutte, Edmonds, Lovász, and others.

After some first,  foundational results, Tutte’s article \cite{Tutte}  stands as a major milestone in the development of graph theory. His theorem on matchings became a  celebrated result in discrete mathematics, and opened  pathways towards  classical and new areas:  linear algebra, probabilistic algorithms, and emerging notions of complexity theory, including Edmonds' definition of good characterisation  and polynomial-time solvability, heralding the dawn  of complexity theory and further, leading to  algebraic and randomized complexity. 

Despite its influence, the concrete arguments of Tutte's paper remain little known, and are often misunderstood.   Section~\ref{sec:Tutte} offers an account: the original proof of the matching theorem, and a few of its ramifications, keeping original terminology whenever possible.  A closer reading of the original text brings some surprises: one of these is that contrary to common belief, Tutte’s proof (Section~\ref{subsec:proof}) is fundamentally combinatorial, 
even though it introduces an algebraic device with lasting impact. 
This algebraic tool is not needed to prove the theorem itself, yet played a crucial role in developing entire fields, like randomized polynomial algorithms, new combinatorial notions or algebraic complexity
 (Section~\ref{subsec:impact}).
Section~\ref{subsec:short} then leaves the historical path and offers a short proof of Tutte’s theorem, with K\H{o}nig's theorem en route, and Edmonds' algorithm as a bonus. 

\medskip
We use standard graph theory terminology and notation, following \cite{SCHRIJVERyellow,LPL} (occasionally \cite{Tutte}, introduced in Section~\ref{sec:Tutte}) with some conventions introduced below or whenever helpful.     

Section~\ref{sec:classical} studies the ``classical'' factor problems, intitiated by Tutte in \cite{factors} to generalize his matching theorem.  Their structure was later  explored by Lovász \cite{structure,structure2}, while  Edmonds was first to provide the corresponding algorithms \cite{E, EJ}. Here we no longer stick to  history, but  aim for an elegant unified view on the subject.  Section~\ref{subsec:parity}  studies minimisation under parity constraints, Section~\ref{subsec:bounds} also adds upper and lower degree bounds.

 Section~\ref{sec:jump} explores the limits of tractability for graph factors and the combinatorial properties that underpin it. The boundary  is the general factor problem  solved by Cornuéjols \cite{corn1}. Section~\ref{subsec:general} presents a short solution  from    \cite{thesis} that also computes the minimum deficiency and is  framed with an eye towards Lovász's request for an extension  to jump systems \cite{LPled}, achieved in Section~\ref{subsec:jump}. 

Graphs may contain parallel edges and loops — both may play essential roles. An edge $uv$ denotes an edge with endpoints $u$ and $v$.
Whenever  $V$ occurs it denotes a finite set, and $n:=|V|$; a  graph under consideration is always $G=(V,E)$ unless stated otherwise; $G$ is bipartite, if $V$ has a partition $\{A,B\}$ so that every edge has one of its endpoints in $A$ and the other in $B$, or equivalently if there is no odd circuit in $G$. 

The degree of $v$ is $d(v)=d_G(v)$, $d_H$ is the {\em degree vector} of a subgraph $H$, or of a subset of edges  $H$; for $X\subseteq V$, $\delta(X)$ is the set of edges with exactly one endpoint in $X$, $d(X):=|\delta(X)|$; for $X, Y\subseteq V$ disjoint sets $\delta(X,Y)$ is the set of edges with one endpoint in $X$, the other in $Y$, $d(X,Y):=|\delta(X,Y)|$.  We put the concerned graph or edge-set in the index in case of ambiguity. The subset of the edge-set $E$ induced by $X\subseteq V$ is denoted by $E(X)$.  

Given a graph, a {\em matching} is a set of edges such that (in this set) every vertex has degree at most~$1$.  Vertices of degree $1$  are said to be  {\em covered}, those of degree $0$ are  {\em missed}.
The maximum matching size is $\nu=\nu(G)$, and $\tau=\tau(G)$ denotes the minimum size of a vertex set meeting every edge. A matching is said to be {\em perfect} if it covers every vertex, that is, if it has $|V|/2$ edges.  
The rest of the paper follows a single guiding thread  progressively relaxing the degree constraints on the subgraphs we seek, and eventually lead to an abstract generalisation.

{\em Paths} and {\em circuits} are assumed to have no repeated vertices; {\em walks} may repeat vertices but not edges.
A path or walk with endpoints $u$ and $v$ is called a $(u,v)$-path or $(u,v)$-walk; $u=v$ is allowed.
{\em Components}  refer  to connected components of graphs. The coordinates (``components'') of vectors will be called {\em entries}.

For $X\subseteq E$, the {\em contraction} $G/X$ identifies  the endpoints of  edges in $X$, deleting these edges.
If edges of $G$ carry weights, the remaining edges of $G/X$ preserve them.

The set of integers, positive integers, rationals, reals is denoted by  $\mathbb{Z}$, $\mathbb{N}$,  $\mathbb{Q}$,  $\mathbb{R}$ respectively, and `+' in the index means restriction to nonnegative numbers;  $[n]:={1,\ldots,n}$.   
For sets $A,B$ the {\em symmetric difference}  is  $A\Delta B:=(A\setminus B)\cup (B\setminus A)$. A set is {\em odd}  if its cardinality is odd, and {\em even} otherwise.  A singleton $\{s\}$ is usually written simply as $s$. 

For a function $w:E\rightarrow \mathbb{R}$ and $X\subseteq E$, $w(X):=\sum_{x\in X}w(x)$.

The entry $(i,j)$ of a matrix means the entry in row $i$, column $j$; for a square matrix $M$, $\det(M)$ denotes the absolute value of its determinant.
For vectors $x,y\in \mathbb{R}^V$,
$\mu(x,y):=\sum_{v\in V} |x_v-y_v|$
is their $l_1$  (also called Manhattan, or taxi) distance. For $X, Y\subseteq\mathbb{R}^V$, 
$\mu(X,Y):=\displaystyle\min_{x\in X, y\in Y} \mu(x,y).$

This article does not attempt to survey the immense landscape of matching-related results; comprehensive accounts of their structural, algorithmic and polyhedral aspects exist in monographs and  handbook chapters \cite{LPL,LPled,bert,pulley} with less space constraints.  Our aim is more focused: to revisit Tutte’s paper, distill some of its core messages, and trace the strength of the ideas stemming from it, following one coherent line of generalisations while pointing to an inevitably incomplete but representative list of ramifications from its origins to the present day. I hope  to share some of the accumulated insights.

\section{Tutte's theorem on matchings}\label{sec:Tutte}

William Tutte played a leading role in breaking the Wehrmacht’s Lorenz cipher during World War II,  an experience that may have helped spark his later interest in combinatorial problems. Among his motivations, he  cites in  \cite{Tutte_book} (Figure~\ref{fig:Tutte-texts} left) a proof of Petersen’s theorem in   \cite{St}. This inspired him to develop his own proof and to generalize it to any graph.  
\begin{figure}[htbp]\label{fig:Tutte-texts}
	\centering
	\begin{minipage}[b]{0.48\textwidth}
		\centering
		\includegraphics[width=\textwidth]{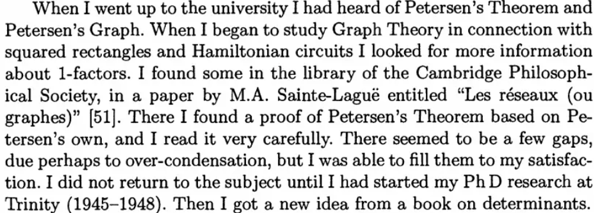}
	\end{minipage}
	\hfill
	\begin{minipage}[b]{0.48\textwidth}
		\centering
		\includegraphics[width=\textwidth]{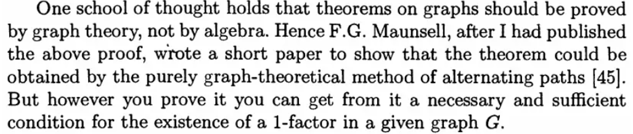}
	\end{minipage}
\caption{Extracts from  \cite[Chapter 3]{Tutte_book}, entitled``Graphs within Graphs''}	
	\label{fig:both_images}
\end{figure}

\vspace{-0.2cm}
We state Tutte's theorem essentially in its original form \cite{Tutte}:

\begin{theorem}\label{thm:Tutte} A graph $G=(V,E)$ does not have a perfect matching, if and only if there exists a set $S\subseteq V$ such that $G-S$ has more than $|S|$ odd components.
\end{theorem}
	
It is routine to derive min–max formulations from existence theorems, and Tutte’s theorem is no exception. However, the following minmax version appeared only later in  \cite{Berge} and is sometimes called the Tutte-Berge theorem. The reader can readily obtain it from Tutte’s theorem.

\begin{theorem}\label{thm:Berge} In a graph $G=(V,E)$ the minimum number of vertices missed by a matching  equals the maximum difference over all $X\subseteq V$ of the number of odd  components of $G-X$ and $|X|$. 
\end{theorem}

 A set $X$ attaining this maximum is called a {\em barrier}.

\subsection{The original proof}\label{subsec:proof}

The purpose of this section is to both present Tutte's original proof and highlight its dual algebraic-combinatorial nature both of which had significant impact. Indeed, as mentioned in the introduction and detailed below,  Tutte used matrix identities to derive  combinatorial statements but it would be misleading to regard the arguments as fundeamentally algebraic,  as the  combinatorial content can in fact be obtained directly.
Our exposition follows  \cite{Tutte} using comments from  \cite[Chapter 3: Graphs Within Graphs]{Tutte_book}.  

Tutte calls a graph {\em prime}, if it has no perfect matching. He calls a vertex $u\in V$    {\em  singular}, if $G-\{u,v\}$ is prime for every $v\in V$.

To each graph  $G=(V,E)$ he associates  the {\em Tutte matrix}  $M_G$,  a $V\times V$ skew-symmetric matrix in indeterminates.    For each  $e=uv\in E$,  one of the entries  $(u,v)$, $(v,u)$ is set to $x_e$ and the other to $-x_e$\footnote{The choice of signs corresponds to orienting the edges. Although the determinant depends on this orientation, since the matrix is skew‑symmetric and each variable appears in two symmetric positions,  det$(M_G)$ is identically $0$ for all orientations or for none.}; for non-edges the entry is $0$.  The determinant det$(M_G)$ is thus a polynomial (of degree $n$) in the variables $x_e$ $(e\in E)$, whose terms correspond to perfect matchings, but cancellation may occur.  The {\em Pfaffian} $P_G$ of a graph $G$  is another polynomial in the same variables having the advantage that each term occurs only once,  and therefore there is no cancellation.  They are a useful tool for enlightening the relation of  det$(M_G)$ and perfect matchings:  skew symmetric matrices satisfy {\em Muir's identity} (1882)   det$(M_G)=P_G^2$ readily impliesg Lemma~\ref{lem:Pf}\footnote{This is sufficient for our purposes; for the definition of the Pfaffian see e,g.~in  \cite{LPL}.}.
All we need about the Pfaffian is Lemma~\ref{lem:Pf} and the identity (\ref{eq:Jacobi}),  taken as as black boxes\footnote{Both are deduced by Tutte from  “Jacobi’s theorem''   presented   with a pointer to  Aitken, {\em Determinants and matrices}, (Edinburgh, 1939), p. 97. Apparently it was a well-known theorem in a well-known book at the time.} : 

\begin{lemma}\label{lem:Pf} A graph  $G$ is prime,  if and only if polynomial det$(M_G)=P_G^2$ is identically zero. \endproof
\end{lemma}
After this observation, Tutte searched among known identities between subdeterminants and found one that proved decisive.  Let $i,j,k,l\in V$ be distinct vertices. Then   
\begin{equation}\label{eq:Jacobi} 
	P_G P_{G-\{i,j,k,l\}} = \pm P_{G-\{i,j\}}P_{G-\{k,l\}}\,\pm\, P_{G-\{i,k\}}P_{G-\{j,l\}}\,\pm\, P_{G-\{i,l\}}P_{G-\{j,k\}}, 
\end{equation}
{\em In particular, if $G$ is prime, it is impossible that exactly one term of the right hand side is nonzero:}
\begin{lemma}\label{lem:Jacobi} If $G$ is prime but both $G-\{i,k\}$,  $G-\{j,l\}$   admit perfect matchings, then 
	
	-- 	 both  $G-\{i,j\}$,  $G-\{k,l\}$ admit perfect matchings, or if not, 
	
	--  both  $G-\{i,l\}$,  $G-\{j,k\}$ admit perfect matchings.\endproof 
\end{lemma}
\prove By Lemma~\ref{lem:Pf} the left hand side of \eqref{eq:Jacobi} vanishes; the middle term   does not vanish by the condition, hence at least one of the other two terms must be nonzero.  \endproof
\begin{lemma}\label{lem:Tutte}  If $G$, $G-\{i,j\}$,  $G-\{j,k\}$ are prime, 
 either $G-\{i,k\}$  is prime, or $j$ is singular. 
\end{lemma}

\prove By the condition,  the conclusion of Lemma~\ref{lem:Jacobi} does not hold for any $l\ne i,j,k$, so at least one of $G-\{i,k\}$,  $G-\{j,l\}$ is prime. If $G-\{i,k\}$ is prime we are done; otherwise, $G-\{j,l\}$ is prime for all  $l\ne  i,k$. Together with the assumptions 
this means that $j$ is singular. \endproof

\noindent{\bf Proof of Theorem~\ref{thm:Tutte}}: Suppose $G$ is prime, and while there exist $u\ne v\in V$ s.t. $G-\{u,v\}$ is prime, join them. The resulting graph $\hat G= (V, \hat E)$ is still prime, it is actually said to be {\em hyperprime}  in  \cite{Tutte},  i.e.   
\begin{equation}\label{eq:sat}
	\hbox{ $uv\in \hat E$ if and only if  $\hat G-\{u,v\}$ is prime.}
	\end{equation}
It is sufficient to prove  Tutte Theorem for  $\hat G$, since any set whose deletion creates sufficiently many odd components in $\hat G$ also does so in $G$.   
We  may therefore assume that $G$ itself is hyperprime. 

Let $S$ be the set of singular vertices. 
From (\ref {eq:sat}) it follows that every $s\in S$ is adjacent to every $v\in V\setminus \{s\}$, moreover, for $i\ne k$ in the same component of $G-S$, $ik\in E$ since otherwise there exists $j$ nonsingular different from them  with $j\notin S$, and then $G-\{ik\}$ is prime by Lemma~\ref{lem:Tutte} and $ik\in E$ by \eqref{eq:sat}.   So each component  of $G-S$ is a complete graph. 

Even components admit perfect matchings internally, and the same holds for all but one vertex in each of the odd components;  $|S|$ of the still missed vertices can be matched to $S$.   Hence,  if the number of odd components of $G-S$ is at most $|S|$, $G$ is not prime. This completes the proof.
\endproof

The above proof  is not entirely self-contained:
it relies on Lemma~\ref{lem:Pf} and identity (\ref{eq:Jacobi}), both algebraic facts about Pfaffians. In fact, only a consequence is  used, Lemma~\ref{lem:Jacobi}, and only in the proof of Lemma~\ref{lem:Tutte}. We show  here that Lemma~\ref{lem:Jacobi},  the key consequence of Tutte's favorite identity \eqref{eq:Jacobi} admits a simple combinatorial proof,  completing Tutte’s argument without algebra: 
\begin{figure}[h]
	\centering
	\includegraphics[scale=0.425]{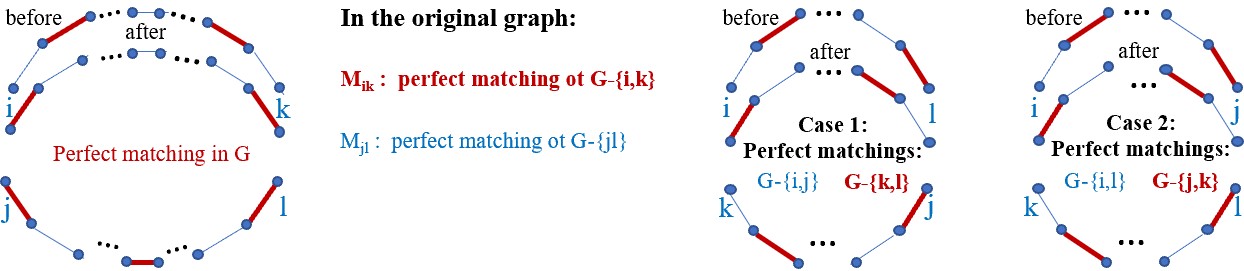}
	\vskip -0.2cm
	\caption{Alternating paths for the Combinatorial version     (of Lemma~\ref{lem:Jacobi}) of Tutte's proof.   }\label{fig:comb}	
\end{figure} 

\noindent {\bf Proof of Lemma \ref{lem:Jacobi}, completing Tutte’s proof of Theorem \ref{thm:Tutte} combinatorially}:

Let $G$ be prime and $M_{ik}$ (red, thick), $M_{jl}$ (blue, thin) perfect matchings of $G-\{i,k\}$,  $G-\{j,l\}$ respectively. In the symmetric difference $\Sigma:=M_{ik}\Delta M_{jl}$ all vertices have  degree $0$, $1$ or $2$; the vertices of degree  $1$ are exactly $i, j, k, l$.  So $\Sigma$ is the disjoint union of circuits, isolated vertices and two alternating paths with endpoints among $i, j, k, l$.

It is impossible, that one path has endpoints  $i$, $k$ and the other  $j, l$,  since both would then be odd paths (Figure~\ref{fig:comb} left):  exchanging along one of them yields a perfect matching of $G$, contradicting the assumption that $G$ is prime.  Therefore either $i$, $l$  and $j$, $k$, or $i$, $j$ and $k$, $l$ are paired by paths.
In either case the paths are even, and exchanging along  one of them we get the claimed perfect matchings.  (Figure~\ref{fig:comb} right shows the two cases.) 
\endproof

 The combinatorial proof is even shorter than the algebraic one. Whether algebraic or combinatorial, however, the argument is far from being intuitive.  Tutte ingeniously unearthed the identity (\ref{eq:Jacobi}) he needed-—but {\em from what planet does such an identity descend?} Lemma~\ref{lem:Jacobi} and its combinatorial proof do not dispel  this mystery!

Several subsequent articles presented combinatorial proofs under the belief that Tutte’s original argument was essentially algebraic. Tutte himself may have contributed to this misunderstanding through remarks in  \cite{Tutte_book} (Figure~\ref{fig:Tutte-texts}, right), delivered with subtle and polite irony. Belck and Maunsell \cite{Belck,M}  in fact observed that ``all reference to determinants can be replaced by elementary arguments". Our aim here was to point out  that Tutte’s proof is already combinatorial: his reference to determinants can be squeezed to  Lemma~\ref{lem:Jacobi} and his relatively lengthy proof of this lemma with determinants can be replaced by a few straightforward graph theoretic lines.

Nevertheless   Lemma~\ref{lem:Tutte} admits a more transparent  interpretation.   In a  1974-75 course in  Budapest, Lovász exploited this point \cite{LLshort}, \cite[Section 3.1]{LPL}  to give a   a natural and elegant combinatorial proof by keeping the above saturation procedure of Tutte, and
restating Lemma \ref{lem:Tutte} as the fact that adjacency becomes an equivalence relation in saturated (hyperprime) graphs after deleting universal vertices. Other elegant short proofs were given by Gallai \cite{Gallai1,Gallai2}.
We present a different short proof  in Section~\ref{subsec:short}.

\subsection{Determinants in Indeterminates}\label{subsec:impact}

Tutte's theorem itself received the attention it deserves.   The algebraic machinery often associated with it, such as determinants and the Pfaffian, has also been credited for its power. 
  As we saw in the previous subsection, these are unnecessary for the proof itself, but subsequent developments were catalysed by these algebraic tools. They opened new directions in theoretical computer science, complexity theory, and even led to fresh combinatorial objects, such as the path matchings introduced by  Cunningham and Geelen \cite{cungeel}.  

A remarkable  property of determinants explains their predictive power  for algorithmic and complexity results as suggested by Lovász \cite{LLRP}.  
{\em Even though their expansion contains  exponentially many terms, their value can be computed in polynomial time via  simple algorithms such  Gaussian elimination; moreover, if a nonzero multivariable polynomial  
	of degree $d$ is evaluated by choosing values for the variables, from a set of cardinality $s$,  the value of the substitution is equal to $0$ with probability $d/s$.} The latter result is the so called  {\em Schwartz-Zippel} Lemma\footnote{Ore (1922), DeMillo and Lipton, Schwartz, Zippel, (1978-79). The assertion  can be proved simply using elementary discrete probabilities, see for instance ``Motwani, Raghavan, Probabilistic Algorithms, Cambridge University Press, 1995". Repeated sampling drives the error probability  exponentially to $0$, placing the concerned problems in the  class RP, sandwiched as  P$\subseteq$RP$\subseteq$NP, commonly considered as close to $P$.}.  

This subsection outlines applications of this idea, first appearing in Tutte's proof. 

\smallskip
\noindent{\bf Bipartite matching}: Let $G=(A,B, E)$ be a bipartite graph, and assume $|A|=|B|$  for simplicity. Define the $A\times B$  matrix $M_{A,B,E}$ whose $(a,b)$-entry  is the variable $x_{ab}$ if $ab\in E$, and $0$ otherwise.  Each variable appears at most once, so no cancellation occurs in the determinant expansion. Hence det$(M_{A,B,E})$ is identically $0$, if and only if $G$ has no perfect matching. 

By evaluating the determinant at random substitutions from a polynomially sized set, one can decide in polynomial time, with exponentially small error probability whether $G$ has a perfect matching. A perfect matching itself can then be found by successive edge deletions. The maximum matching problem in bipartite graphs reduces to this procedure.

 In the bipartite case $M_G = 
	\begin{pmatrix}
			0 & M_{A,B,E} \\
			-M_{A,B,E} & 0  \\
 	\end{pmatrix}$, so 
   $P_G=$det$M_{A,B,E}$, and  Muir's identity  reduces to the trivial equality
   det$(M_G)=$(det$(M_{A,B,E}))^2$.
   
  Although Tutte’s matrices  subsume this bipartite special case, the bipartite construction is often referred to as the {\em Edmonds matrix}. 

\smallskip
\noindent{\bf Matching, Matroid Matching}:  
For general  graphs, the randomized approach must replace  the Edmonds matrix by the Tutte matrix $M_G$
and Lemma~\ref{lem:Pf} (via Muir’s identity) becomes essential. The randomized algorithms are not faster though than the exact ones. However,  Lovász \cite{LLRP} extended this method to matroid matching, obtaining a randomized algorithm  that is  substantially simpler and  faster than the exact algorithms. Consequently, both matching and matroid matching lie in RP even before invoking more sophisticated polynomial-time algorithms. Geelen \cite{G} later derandomized this approach for graphs.

\smallskip
\noindent{\bf Exact Matching}: The {\em Exact Matching problem } asks, given a graph $G=(V,E)$, 
a subset $R\subseteq E$
and an integer $k\in\mathbb{N}$, whether there exists a perfect matching $M$, with $|M\cap R|=k$. 
Its deterministic complexity remains open, even for bipartite graphs. 

 Using the Tutte matrix, however, a randomized polynomial algorithm can be designed \cite{MVV}.  Mulmuley, Vazirani and Vazirani complete  Lovász's approach with an additional idea.
 
 A naive attempt would substitute in $M_G$  the same variable $y$  for each edge in $R$, the value $1$ for edges of $E\setminus R$, and $0$ otherwise.  By Muir's identity and  interpolation one can test whether the coefficient of $y^k$  in  $P_G$ is nonzero. However, since different edges now carry the same variable, cancellation may occur in the Pfaffian expansion, and  Lemma~\ref{lem:Pf} no longer applies to individual coefficients.
 
Mulmuley, Vazirani, and Vazirani avoid this difficulty by assigning random edge weights. thereby {\em isolating} a unique minimum weight perfect matching $M$, $|M\cap R|=k$,  with high probability. Multiplying the variable $y$ of the naive approach by these edge weights   prevents cancellation in the Pfaffian expansion, with high probability. 

In planar graphs, {\em Pfaffian orientations} (see the definition e.g. in  \cite{LPL}) yield an exact polynomial-time algorithm \cite{V}, bypassing randomisation altogether: when the Tutte matrix is defined from such an orientation (exploiting the freedom  in its definition, see Section~\ref{subsec:proof}) all terms in the Pfaffian or the determinant expansion have the same sign, so the naive substitution suffices and the isolation mentioned above is unnecessary. 

Generalisations, such as the minimum odd 
$T$-join problem -- and equivalently the minimum odd cycle problem with $\pm 1$ weights in graphs without negative cycles -- are also solvable in randomized polynomial time \cite{GK}, but their deterministic complexity remains open\footnote{Polynomial-time solvability for planar graphs was recently observed by Csáji, Bérczi-Kovács, Jüttner, and Seb\H{o}.}. These  may be easier than Exact Matching, but some variants have surprisingly turned out to be \NP-hard. (For the current state of the art and new positive and negative results, see  \cite{SchS}.) They also contain the max cut problem for graphs embedded in the projective plane. These are not even solved in the special cases satisfying the conditions for idealness of Seymour's conjectures on ideal binary clutters (see eg.~\cite[Chapter 78, and 80.4, 80.5, Problem~62]{SCHRIJVERyellow}). 

\smallskip
\noindent{\bf Path Matching}: Cunningham and Geelen \cite{cungeel} introduced path matchings, a generalisation of matchings and common independent sets of two matroids. They solved their weighted optimisation problem in polynomial time, determined the linear descriptions of associated polyhedra, and proved their total dual integrality\footnote{TDI is a central notion of combinatorial optimisation essentially meaning the presence of a  combinatorial minmax theorem  for all weights, cf. e.g.~\cite{SCHRIJVERyellow}.}. Their work provides a perfect illustration of the predictive power of Tutte matrices: the fact that Path Matching lies in RP already suggested polynomial-time solvability. Our explanations concern the graphic special case.

For subsets $I, J\subseteq V$ an $(I,J)$-{\em path matching} is a set of edges $M$ so that the components of the graph $(I\cup J, M)$ are paths with one endpoint in $I$, the other in $J$, and internal vertices in $I\cap J$.  The {\em value} is the number of its edges, counting the single-edge components induced by $I\cap J$   twice. 

In response to a question of Cunningham and Geelen, Lovász observed that path matchings correspond to suitable submatrices of the Tutte matrix. Consequently, computing their value lies in RP.  A simple linear algebraic lemma enables a recursive computation of the rank of $M_G(I,J)$ and leads to a minimax formula  shown by Cunningham and Geelen \cite{cungeel}, combinatorially,  to equal the maximum value of a path matching. The underlying combinatorial insight is  an ingredient of the short proof of Tutte's theorem in Section~\ref{subsec:short} (cf. Lemma~\ref{lem:path}). 

\begin{lemma}\label{lem:key} Let $A$ be a matrix over a commutative ring.   If deleting both row $i$ and column $j$  decreases the  rank of $A$, then  deleting row $i$ alone or column $j$ alone already decreases the rank. 
\end{lemma}

\prove Assume to the contrary  that both matrices $A_i$, obtained by deleting row $i$  and $A^j$ obtained by deleting column $j$  have  the same rank as $A$. Then   row $i$ is a linear combination  of the other rows, and column $j$  is a linear combination of the other columns.

 Consequently, the $i$-th row of $A^j$ is still a linear combination of the other rows. Thus deleting row $i$ of $A^j$  does  not decrease the rank, a contradiction. \endproof

This lemma can be immediately applied to  $i\in I\Delta J$, say $i\in I\setminus J$, because any edge $e=ij$ incident to $i$ introduces a variable $x_e$ that appears only in the $(i,j)$-entry and therefore cannot cancel. Therefore, the condition of Lemma~\ref{lem:key} is satisfied. 

As long as $I\Delta J\neq\emptyset$, this allows a reduction to a smaller instance; if $I\Delta J=\emptyset$, we have   $I=J$ and Tutte’s theorem applies.

 In this way, Tutte's theorem extends to path matchings, establishing a good characterisation for their maximum value. 
 Interpreting this NP$\cap$coNP characterisation and membership in RP as evidence of polynomial-time solvability, Cunningham and Geelen were then able to prove that the maximum path matching problem is indeed solvable in polynomial time.

\smallskip
\noindent{\bf Algebraic Complexity}:      
The use of determinants in matching theory foreshadowed the notion of algebraic complexity. Determinants are polynomials with exponentially many monomials whose values can nevertheless be computed efficiently. This phenomenon is formalized via algebraic circuit complexity, with complexity classes such as VP and VNP, analogues of P and NP.

A recent breakthrough by Srinivasan \cite{algebraic} constructs explicit multilinear polynomials separating ``monotone VP'' from ``monotone VNP'' by demonstrating that a sequence of polynomials with nonnegative coefficients can require exponential-size monotone agebraic circuits.  

Algebraic complexity theory,  beyond formalizing the phenomenon that polynomials of enormous apparent complexity may admit efficient evaluation, also furnishes a conceptual laboratory in which theorems such as VP$\ne$VNP play a role analogous to   the classical, legendary, mysterious, still open  ``Is P$=$NP ?'' question.

\smallskip
\noindent{\bf Enumeration}: Lovász's method extends from decision to counting problems. When a graph admits a Pfaffian orientation (see above), the number of perfect matchings can be computed exactly by evaluating a determinant. Determining when such orientations exist is a natural and central open question \cite[Section 8.3]{LPL}.

Counting perfect matchings arises in statistical physics and motivates several open problems surveyed in  \cite[Chapter 8]{LPL}. One long-standing conjecture asserts that regular graphs in which every edge lies in a perfect matching have exponentially many perfect matchings; this has been proved for cubic graphs \cite{enum}. 

\medskip
This section has chosen to focus on one kind of application of Tutte’s original work on matchings. Entirely graph-theoretical consequences of  Tutte’s theorem, such as those explored by Szigeti \cite{Sz}, form a rich direction that lies beyond our scope, but can be found in the literature.

\subsection{A short combinatorial proof}\label{subsec:short}

We finish Section~\ref{sec:Tutte} with a combinatorial proof of Tutte's theorem.
Our version   is inspired  by Lemma~\ref{lem:key} specialized to an elementary combinatorial statement, and is also influenced by Edmonds' algorithm \cite{E} as presented in  \cite[Exercises 7.33–7.34]{exercises}.   The following  proof, of  the so called Tutte-Berge minmax form Theorem~\ref{thm:Berge}, combines these ideas in a simple way: 
 
\begin{lemma}\label{lem:path} Let $G$ be a graph, and $uv\in E(G)$. Then either
	$\nu (G-u)<\nu (G)$, or  $\nu (G-v)<\nu (G)$, or else  for any maximum
	matching $M_u$  of $G-u$, and $M_v$ of $G-v$: $M_u\cup M_v$
	contains a $(u,v)$-path $P$ alternating  between $M_u$ and
	$M_v$.
\end{lemma}

\prove Suppose $\nu (G-u)=\nu (G)=\nu (G-v)$. Then $M_u$ covers $v$, and $M_v$ covers $u$, otherwise the edge $uv$ could be added for a larger matching of $G$. 

Since all degrees of $M_u\cup M_v$ are $0$, $1$ or $2$, this union is the disjoint union of alternating paths.
Moreover, {\em the other endpoint of the alternating path containing $u$ is $v$}: if it were  some  $x\ne v$, then exchanging  on it, we get again a matching of maximum size missing both $u$ and $v$. (Figure~\ref{fig:paths}).
\endproof
\begin{figure}[h]
	\centering
	\includegraphics[scale=0.4]{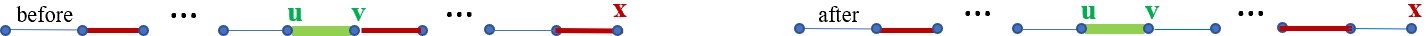}
		\vspace*{-0.1cm}
		\caption{$M_u$:red(thick); $M_v$:blue(thin); $uv$:green; left: before exchange; right: after augmenting }
	\label{fig:paths}
\end{figure} 	
Bipartite matchings  are discussed in  \cite{Cameron}. We only state K\H{o}nig’s theorem \cite{Konig} in passing:
\begin{theorem} If  $G=(V,E)$ be a bipartite graph, then $\tau (G)=\nu
	(G)$.
	\end{theorem}

\prove The inequality $\tau (G)\ge\nu
(G)$ is straightforward. For the reverse inequality, we proceed by induction. Let $uv\in E$ be arbitrary. Either deleting $u$ or deleting $v$ decreases $\nu$ -- say $\nu (G-u)<\nu (G)$ -- since otherwise there exists an even $(u,v)$ path by Lemma~\ref{lem:path}, contradicting bipartiteness. Then  $\tau(G)\le \tau(G-u)+1\le \nu(G-u)+1\le\nu(G)$, where we also used the theorem for $G-u$.\endproof 

\noindent{\bf Short proof of Theorem~\ref{thm:Berge}}:
Let $uv\in E$ be arbitrary. If $\nu (G-u)<\nu (G)$,  we are done: a set  $X\subseteq V\setminus\{u\}$ satisfying the min-max equality in $G-u$ extends to $X\cup\{u\}$  for $G$. 
 
 Otherwise, if  $\nu (G-u)=\nu (G)=\nu (G-v)$, Lemma~\ref{lem:path} provides an alternating path $P$ (which becomes a circuit if the edge $e$ is added). In $G/P$ the number of missed vertices by any matching $M$ does not increase upon expanding  $P$: whether or not  the contracted vertex is matched by $M$,  one can extend  $M$ by adding a perfect matching of the circuit  $P\cup \{uv\}$ removing the vertex which is matched outside $P$, or an arbitrary vertex if unmatched. \endproof

The last reasoning of the proof also explains the correctness of  Edmonds’ algorithm \cite{E,exercises}[Exercises 7.33–7.34]: unshrinking an odd cycle, the number of missed vertices remains, and  barriers remain barriers.

\section{Classical factors}\label{sec:classical} 

In a subsequent article \cite{Tutte_factor},   Tutte derived necessary and sufficient conditions  for the existence of subgraphs with prescribed degrees not necessarily equal to~$1$. His methods and results naturally extend to more general settings, allowing upper and lower degree bounds and parity conditions, and are extendable to orientations under degree constraints. 

In this section we present these  generalisations in a unified framework: first parity-restricted subgraphs, then degree-bounded subgraphs, and finally their simultaneous occurrence   with possible extensions to orientations.  Since the  relevant definitions are most naturally introduced together, we adopt a broader perspective when we introduce the relevant notions 
in preparation for the abstract generalisation of the next section,  which builds on the results established here.  

We consider subgraphs whose degree vectors are required to lie in a  non-empty set $H\subseteq \mathbb{Z}_+^V$. Such subgraphs are called {\em $H$-factors}. Even if  a $H$-factor does not exist, we are interested in $F\subseteq E$ minimizing the distance $\mu(d_F, H)>0$.  We simplify the notation: $\mu(F, H):=\mu(d_F, H)$  is also called the {\em deficiency} of $F$ (with respect to $H$); $F\subseteq E$ attaining this minimum is {\em ($H$-)optimal}.  The degree constraints  are described independently at each vertex:  $H$ is defined as the Cartesian product of sets of one-dimensional vectors, that is, of sets of numbers. 

\medskip
We say that $H\subseteq \mathbb{Z}^V$ is a  {\em sponge},   if it is the Cartesian product of nonempty subsets  $H(v)\subseteq\mathbb{Z}_+$, $v\in V$,    
and each $H(v) (v\in V)$ is   a one-dimensional {\em sponge}, i.e. 
\vspace{-0.2cm}
	\[\hbox{ if } a,b\in H(v), \hbox{ and }  a< i< b \hbox{ with } i\notin H(v), \hbox{ then }   i-1,  i+1\in H(v), \vspace{-0.2cm}\]     
 i.e. there are no ``gaps'' of at least two consecutive integers in $H(v)$.
 We write $l$ and $u$ for the vector of entrywise minima and maxima of $H$.
 
 For a graph $G=(V,E)$  a sponge (called  a ``prescription'' in  \cite{factors}) is restricted to be $H\subseteq  \mathbb{Z}_+^V$, and then   $0\le l\le u\le d$, where   $d\in\mathbb{Z}_+^V$ is the degree function. 
 
A {\em gap} of  $K\subset \mathbb{Z}$ is a closed interval  $[i,j]$, $i, j\in\mathbb{Z}, i\le j$ such that $K\cap [i, j]=\emptyset$, and $i-1, i+1\in K$; its length is $j-i+1$. 

It is essential to restrict ourselves to sponges \cite[Section 10.2, paragraph after Exercise 10.2.2]{LPL}: the $H$-factor problem is already \NP-hard in cubic graphs when $H(v)$ is $\{0,3\}$ or $\{1\}$ for each $v\in V$. In contrast, when $H$ is a sponge, a rich structural theory exists: degree-constrained factors with parity conditions and bounds were analysed in  \cite{structure}, and extended to $H$-factors where $H$ is a sponge \cite{factors}.  However,  for some time  {\em there was no  a polynomial algorithm  for  sponges, with the existence question of $H$-factors  explicitly raised for sponges in  \cite{factors}},  and repeated in  \cite{LPL}.   

A set of integers is an {\em interval} if it contains all integers between its minimum and maximum, and a {\em parity sponge} if it contains every second integer in this range. Intervals and parity sponges are referred to as  {\em classical} sponges. 
A sponge   $H\subseteq \mathbb{Z}^V$ is called {\em interval sponge},  {\em parity sponge} or {\em classical sponge} if   $H(v)$ is of the corresponding type for each $v\in V$; the associated factors are  named accordingly. 
Classical sponges may thus have both  {\em parity-vertices}, and {\em interval-vertices}. 

It is convenient to describe a classical sponge $H$ with the triple  $(l,u,\Pi)$, where $l, u:V\rightarrow \mathbb{Z}_+$ are as before, and $\Pi\subseteq V$ is the set of parity vertices. We assume without loss of generality that $l(p)\equiv u(p) \mod 2$    for all $p\in \Pi$, and that $\{v\in V:l(v)=u(v) \}\subseteq \Pi$; the corresponding sponge is then denoted by $H_{(l,u,\Pi)}$,  and $H_{(l,u,\Pi)}$-factors are also called {\em $(l,u,\Pi)$-factors}. 

The following theorem provides a good characterisation for the existence of classical factors. It is essentially the same as the main result of  \cite{Tutte_factor} or of  \cite{subgraphs} for interval sponges; a polynomial algorithm  follows from  \cite{EJ}.   We present an algorithmic proof in Section~\ref{subsec:bounds}.  For simplicity we state  only the existence version but the minimisation of the deficiency can be deduced with the same method. The ultimate generalisation in Section~\ref{sec:jump} addresses deficiencies directly.   

 \begin{theorem}\label{thm:factor}
No  $(l, u, \Pi)$-factor exists if and only if  there exist disjoint $L, U\subseteq V$ such that 
 \[ l(L) + d_G(L,U) + \omega(L,U) > u(U)+ \sum_{x\in L} d_G(x), \]
 where $\omega(L,U)$ is the number of $(L,U)$-{\em odd}  components $C$ of $G-(L\cup U)$, i.e.~components satisfying $C\subseteq\Pi$   for which the parity of $l(C)\equiv u(C)\,\mod~2$ differs from that of $d(C,L)$.  
 \end{theorem} 
\vspace{-0.15cm}
We first verify  the so called ``easy direction''.
Suppose that an $(l, u, \Pi)$-factor $F$ exists, and let $L, U\subseteq V$ be disjoint. We first  estimate from above the sum of the degrees in  $F$ on $U$ together with the sum of  the degrees in $E\setminus F$ on  $L$: 
\vspace{-0.2cm}
\[\sum_{y\in U}d_F(y) + \sum_{x\in L}  (d_G(x) - d_F(x))\le  u(U) - l(L)  +\sum_{x\in L} d_G(x). \]
\vspace{-0.2cm}  
The left hand side can be written as
 \[2 |E(U)\cap F|+ d_F(U,V\setminus (U\cup L))+ d_F(U,L) +  2 |E(L)\setminus F|+ d_{E\setminus F}(L,V\setminus (U\cup L))+ d_{E\setminus F}(L,U)\ge \]
 \vspace{-0.49cm}
\[d_F(U,V\setminus (U\cup L))+d_{E\setminus F}(L,V\setminus (U\cup L)) + d_G(U,L)\ge \omega(L,U)  + d_G(U,L),\] 
where in the first inequality $|E(U)\cap F|,\, |E(L)\setminus F|\ge 0$,  $d_F(U,L)+d_{E\setminus F}(L,U)= d_G(L,U)$ are used. Moreover, each $(L,U)$-odd component of $G-(L\cup U)$ must contribute at least one edge either of $F$ to $U$, or  of $E\setminus F$ to $L$, yielding the final inequality. We conclude by  
\[\omega(L,U)  + d_G(U,L)\le u(U) - l(L)  +\sum_{x\in L} d_G(x).\]
Thus the inequality of the theorem is violated whenever an $(l, u, \Pi)$-factor exists, completing the  direction of  proof under consideration. In fact, the difference of the left and right hand sides is a lower bound for the deficiency  $\mu(F, H_{(l,u,\Pi)})$, and in Section~\ref{subsec:bounds} we present $F\subseteq E$ with equality here. We in fact outline a  unified approach   covering   factorisation and orientation problems under degree constraints, including  algorithmic aspects. Theorem~\ref{thm:factor} will be proved as a representative example; by specialisation this gives  yet another proof of Theorem~\ref{thm:Berge}, and contains min-max, structural and algorithmic consequences to classical factors outlined in Section~\ref{subsec:bounds}.

\subsection{Parity}\label{subsec:parity}

 In this subsection we  consider the  basic  special case of graph factors  when $H(v)$ $(v\in V)$ consists of every second integer in the interval $[0,d(v)]$. Then   $H$ is determined by a set $T\subseteq V$: $H(t)$ is the set of odd integers of $[0,d(t)]$ for $t\in T$, and the set of even ones for $v\in V\setminus T$; an $H$-factor is then a {\em $T$-join}.
   Allowing vertices $u\in U\subseteq V$ with $H(u)= [0,d(u)]$  does  not lead to a genuine extension: deleting the edges between vertices of $U$ and   identifying $U$ reduces existence and optimisation to that of $T$-joins.    

 There exists a $T$-join if and only if $|T|$ is even. The existence of degree-bounded $T$-joins is  discussed in the next section, and they  will serve as tools for the most general factors.
 
  A cut $\delta(X)$ $(X\subseteq V$) is a $T$-cut, if $|X\cap T|$ is odd. Every $T$-cut intersects every $T$-join, implying $\tau(G,T)\ge\nu(G,T)$, and $\tau(G,T)\ge\nu_2(G,T)/2$, where  $\tau(G,T)$
is the mininum size of a $T$-join,   $\nu(G,T)$  the maximum number of disjoint $T$-cuts, and  $\nu_2(G,T)$  the maximum number of $T$-cuts covering each edge at most twice. 

A convenient ``language'' and main tool for $T$-joins is provided  by {\em conservative weight functions}, -- i.e. those with $w(C)\ge 0$ for every circuit $C$ --, and their {\em distance functions} in undirected graphs.  
For (not necessary nonnegative) $w:E\rightarrow \mathbb{R}$, and $X\subseteq E$,  $w[X]:E\rightarrow \mathbb{R}$ denotes the following function: 
\[w[X] (e):=
\begin{cases}
	\,\,\,\,\, w(e) \hbox{ if }  e\in E\setminus X,\\
	-w(e) \hbox{ if }  e\in X
\end{cases}
\]
For any $F\subseteq E$, $T_F\subseteq V$ is the set of vertices of odd degree of $F$. Guan \cite{Guan} observed: 

\begin{fact}  $F\subseteq E$ is a $w$-minimum $T_F$-join, if and only if $w[F]$ is conservative.  
\end{fact}

\prove For any circuit $C\subseteq E$, $w[F](C)=w(C\setminus F)-w(C\cap F)=w(C\triangle F)-w( F)\ge 0$, since 
$C\triangle F$ is a $T_F$-join  $F$ is a minimum  $T_F$-join. 
\endproof

We henceforth restrict ourselves to weights   $w:E\rightarrow\{1,-1\}$. General integer or rational weights reduce to this case by subdivision and scaling, and contraction of $0$-weight edges,  preserving polynomial-time solvability; deducing weighted versions of the structural theorems requires some more observations, but the extensions are simple \cite{pot}. 
We therefore restrict ourselves here to $\pm1$-weights which are completely sufficient for studying ``potentials'' in undirected graphs, describing  structural properties of $T$-joins, including Tutte's theorem that we deduce.

\begin{figure}[h]
	\centering
	\includegraphics[scale=0.7]{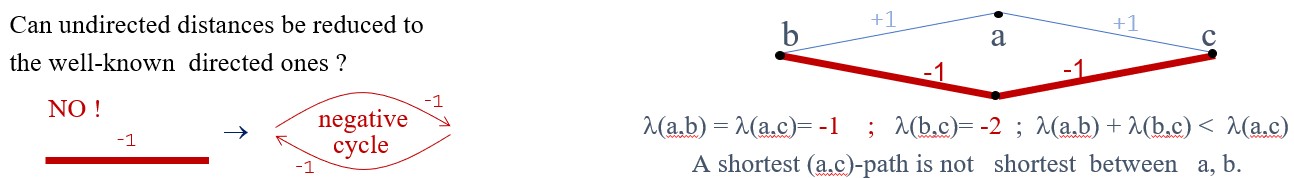}
	\caption{Distance function violating the triangle inequality, the subpath property, etc.  }
	\label{fig:shortest}
\end{figure} 

Let $\lambda(x,y):=\lambda_w(x,y):=\min\{w(P):\hbox{$P$ is an $(x,y)$-path}\}$,  $x_0\in V$ arbitrary,  $\lambda(x):=\lambda(x_0,x)$. If  $w$ is conservative, it is useful to know that $\lambda_w(x,y):=\min\{w(P):\hbox{$P$ is an ${x,y}$-join}\}$. Let $E^-:=\{e\in E: w(e)=-1\}$.  The following fact, lemma and theorem are from  \cite{JCTB, pot}:

\begin{fact}\label{fact:distances} Let $w:E\rightarrow\{1,-1\}$ be conservative.  Then 
	
	\begin{itemize}
	\item [{\rm(i)}]  For each $e=uv\in E$: $|\lambda(u)-\lambda(v)|\le 1$ and if $G$ is bipartite the equality holds. 
		\item [{\rm(ii)}] Let $b\in V$, $\lambda(b)=\min\{\lambda(v): v\in V\}$. Then $|\delta(b)\cap E^-|=1$, unless $b=x_0$, when it is $0$.  
	\item [{\rm(iii)}]  If $C$ is a circuit, $w(C)=0$, then $w[C]$ is  conservative, and $\lambda_w=\lambda_{w[C]}$.
\end{itemize}
\end{fact}

\prove {\rm (i)}: Let $P$ be a shortest $(x_0,v)$-path.  Then $P\Delta {e}$ is an
$\{x_0,u\}$-join. {\rm (ii)}: Let $Q$ be a shortest $(x_0,b)$-path. If the last edge $e=ab$ of $Q$ were not negative, $Q\setminus \{e\}$ would contradict  the definition of $b$. If there were another negative edge $f=bb'$ incident to $b$, then $P':=P\cup {f}$ would be an $\{x_0,b'\}$-join, $w(P')<w(P)$, contradicting the definition of $b$ again. 

\noindent
 {\rm (iii)}: If $Q$ is an arbitrary circuit, then $w[C](Q)=w(Q\setminus C)-w(Q\cap C)=w(Q\triangle C)-w(C)\ge 0-0$, since  $Q\triangle C$ is a disjoint union of cycles, so $w[C]$ is conservative.  Similarly, if $Q$ is an $(x,y)$-path, $x,y\in V$, then $Q\triangle C$ is an $\{x,y\}$-join and  $w[C](Q)\ge w(Q\triangle C)$, so choosing $Q$ to be $w$-shortest,   $\lambda_w(x,y)\ge\lambda_{w[C]}(x,y)$. Applying this to $w':=w[C]$ and using $w[C][C]=w$ we get $\lambda_w(x,y)\ge\lambda_{w[C]}(x,y)$, and therefore the claimed equality holds for any pair of vertices. 
\endproof

It is better to use $T$-joins with $|T|=2$ instead of paths to avoid case distinctions according to whether an appended vertex is contained in $P$ or not. This is also  essential in shortest path computations in undirected graphs (Figure~\ref{fig:shortest}). If the reader needs more details of   proofs they can be found in  \cite{JCTB} or  \cite{pot} (generalized to weights), for proving the following Lemma and Theorem:

\begin{lemma}\label{lem:shortest}
	Let $G$ be bipartite, and $w$, $b$ as in Fact~\ref{fact:distances}. Then the contraction of $\delta(b)$ preserves conservativeness and the distances from $x_0$, except   if $b=x_0$, when $\lambda (b^*, v)=\lambda (b, v)-1$ for all  vertices $v$ different from $b^*$, where $b^*$ is the new vertex arising by the contraction of $\delta(b)$.  
\end{lemma}

\prove (Sketch) If $C$ is a $0$ weight circuit or a shortest path from $x_0$, then $w(C\cap \delta(b))=0$, because $C$ is either disjoint from $\delta(b)$, or $C\cap \delta(b)$ consists of a positive and a negative edge, because of Fact~\ref{fact:distances} (iii) and (ii).  
 \endproof

We are now ready to state the main structure theorem about undirected distances \cite{JCTB}.

 Let $G$, $x_0$, $w$, $\lambda$ as before ($w$ is conservative), $m:=\min\{\lambda(x):x\in V\}$, $M:=\max\{\lambda(x):x\in V\}$,  and denote $\Dscr^i= \Dscr^i(\lambda)$ the (vertex-set of) the components of the graph $G(V^i)$ induced by the {\em level-set} $V^i:=\{v\in V: \lambda(v)\le i\}$, and $\hat \Dscr^i=\hat \Dscr^i(\lambda)$ those of $G(V^i)-\{uv: \lambda(u)=\lambda(v)=i\}$.
 
 \smallskip
 \noindent {\bf Note}: Both $\Dscr^i$ and $\hat \Dscr^i$ partition $V^i$, $(i=m, m+1, \ldots, M)$;     \cite{pot} extends this to weights.

 \begin{theorem}\label{thm:shortest} 
If $D\in\Dscr\cup\hat\Dscr$, then $|\delta(D)\cap E^-|=1$ if $x_0\notin D$, and $|\delta(D)\cap E^-|=0$, if  $x_0\in D$.
 \end{theorem}
 
 \prove (Outline) Subdivide every edge $e=xy\in E$ into two with a new vertex $v_e$ to obtain the bipartite graph $\hat G=(\hat V,\hat E)$, and  $\hat w(xv_e):=\hat w(v_ey):= w(e)$; $\hat w$ is conservative, and the distances between the original vertices are   {\em exactly doubled}; keep the notation for the original vertices and $\hat\lambda$ denotes the distance function from $x_0$ in $\hat G$ according to $\hat w$, and apply Fact~\ref{fact:distances}(i) to $\hat w$. (If the two endpoints of an edge $e\in $ are at the same level, it is easy to see that $v_e$ is one level higher.) 

We can then suppose without loss of generality that  $G$ is bipartite.
Now the theorem follows by induction applying Lemma~\ref{lem:shortest}: indeed,  contract  $\delta(b)$, $b\ne x_0$ while possible, and set $b=x_0$ when this is the unique choice. Using the induction hypothesis, and then adding $\{b\}$ to $\Dscr$ of the contracted graph,   the theorem follows from Fact~\ref{fact:distances}(ii).
\endproof 

To show the unifying role of undirected distances we deduce Theorem~\ref{thm:Berge}. (In fact,  the proof can be completed to decuce the Edmonds-Gallai  structure theorem \cite{JCTB}.)

\smallskip
\noindent {\bf Another proof of Theorem~\ref{thm:Berge}}:
The number of missed vertices is clearly at least odd$(X)-|X|$ for all $X\subset V$. 

To prove that there exists $X\subset V$ for which the equality holds, let $M$ be a matching, and define $w(e)= -1$ if $e\in M$, $w(e)=1$ if $e\in E\setminus M$.  Then add a new vertex $x_0$ to the graph and $\{x_0v: \hbox{ for all } v\in V\}$, with $w(x_0v)=-1$ if $v$ is missed by $M$, and $1$ otherwise. A circuit  contains two consecutive negative edges if and only if the common endpoint of these is $x_0$ and the rest of the edges of the circuit form an augmenting alternating path.  Therefore, $w$ is conservative if and only if $M$ is a maximum matching.

Suppose now that $M$ is a maximum matching, i.e. $w$ is conservative, and let   $D^0\in\hat\Dscr^0$, $x_0\in D^0$, and $X:=\{x\in D^0\setminus \{x_0\} : \lambda(x)=0\}$.

It can be readily checked that $X$ is a barrier.   
\endproof

We finally deduce from Theorem~\ref{thm:shortest} Lovász's and Seymour's minimax theorems  giving a good characte\-ri\-sation of minimum factors with parity constraints:    

 \begin{corollary}\label{thm:Seymour} 
	If $G$ is bipartite $\tau(G,T)=\nu(G,T)$. 
\end{corollary}
\prove Let $F$ be a minimum $T$-join, i.e.~$|F|=\tau$,  and $w:=\underline 1[F]$. From Fact~\ref{fact:distances}(ii) and bipartiteness,  $|\lambda(u)-\lambda(v)|= 1$, so for each negative edge $e$ there exists exactly one $D_e\in\Dscr$, $e\in\delta(D_e)$, and by  Fact~\ref{fact:distances}(iii), $e$ is the only negative edge of $\delta(D_e)$. Therefore $\{D\in\Dscr: x_0\notin D\}$ has $|F|=\tau$ elements, and they are clearly disjoint. \endproof

 \begin{corollary}\label{thm:Lovasz} 
$\tau(G,T)=\nu_2(G,T)/2$. 
\end{corollary}

\prove Let $F$ be a minimum $T$-join, i.e.~$|F|=\tau$,  and $w:=\underline 1[F]$, and then subdivide every edge into two edges having the original weight, with a new point, and apply Corollary~\ref*{thm:Seymour}.\endproof

Other applications of Theorem~\ref{thm:shortest} are surveyed in  \cite{Fra96}; sharpenings of theorems of Lovász and Seymour occur in  \cite{FST};  consequences of Theorem~\ref{fact:distances} for multicommodity flows or $T$-cuts  with the max-flow min cut property are shown in  \cite{thesis, FSz}. 
 
The  Edmonds-Gallai struture theorem \cite{LPL} is a direct special case of Theorem~\ref{thm:shortest}, which actually has the straightforward consequence of generalizing it to weighted matchings, and  also to weighted $T$-joins. Its simple statement is due to the unifying effect of the generalisation. The generalisation to $T$-joins (or equivalently, undirected distances) has a simplifying effect: different cases of  Lovász's  ``structure algorithm'' based ``upon the Gallai-Edmonds structure theorem'' \cite[9.3]{LPL} get unified. 

The sets $\Dscr$ and $\hat\Dscr$ can of course be computed in polynomial time as we saw about  distances. It is more interesting to note that the above proof can be turned to an algorithm: defining $\Dscr$ and $\hat\Dscr$ in each step of the algorithm for the current $T$-join, and keeping an $(x_0,x)$ path for every vertex $x\in V(G)$, as long as  both statements of Theorem~\ref{thm:shortest} are not satisfied,  either a negative cycle is detected, i.e.~the current $T$-join can be improved,  or some of the paths and therewith the level-sets can be improved  \cite{thesis,finding}. An optimal $T$-join can be found in this way in $O(n^4)$ time. 
 
\subsection{Adding bounds}\label{subsec:bounds}
 
Tutte \cite{Tutte_factor} originally assumed $|H(v)|=1$ for each $v\in V$. Lovász \cite[Section 10.2]{subgraphs,LPL}  considerably broadened this framework by allowing arbitrary intervals in $H$, developing min-max theorems for the deficiency, and characterizing the structure of $H$-factors. Then allowing parity constraints in  intervals for an arbitrary subset of vertices does not substantially complicate the results, proofs, or algorithms. The corresponding {\em classical factors} then admit simple, parity-based good characterisation, structural description and algorithms, and the goal of this section is to show general ideas that solve a large variant of classical factor problems.
 
Tutte’s factor theorem \cite{Tutte_factor} and  more generally, Lovász’s min–max and structure theorems \cite{subgraphs}, the related optimisation problems get a general algorithmic and polyhedral explanation   by  Edmonds and Johnson \cite{EJ}.
Even though this work of Edmonds and Johnson does not contain proofs, this is where bidirected graphs appear in the study of graph factors, and were mentioned to the author by Jack Edmonds in Bonn in 1985. 

We sketch here an elementary, combinatorial way of deducing the results on classical factors, working out the interesting properties of bidirected "reachability",  that leads to a ``structure algorithm'', not using blossoms, similar to the one presented for $T$-joins, and also containing Lovász's structure algorithm for matchings.    
``Minty's theorem'', see eg. Lovász \cite[Exercise 6.10]{exercises} generalizes to bidirected graphs as follows.

{\em Let $G=(V, A)$ be a digraph, and $s,t\in V$. Then either there exists a directed $(s,t)$ path, or an $(s,t)$-dicut, i.e. $S\subseteq V, s\in S, t\notin S$, with all edges of  $\delta(S)$ directed towards $S$; the set  of vertices reachable from $s$ by a directed path determines the unique inclusionwise minimal $(s,t)$-dicut. } \endproof

We now outline how to extend this fact to bidirected graphs and apply the  extension  to  classical factors. 
Let  $\alpha,\beta\in \{1, -1\}$ and  define 
\[ a^{\alpha}b^{\beta}(x) := 
\begin{cases} 
	\alpha & \text{if } x = a \\
	\beta & \text{if } x = b \\
	\alpha + \beta & \text{if } x = a = b \\
	0 & \text{otherwise.}
\end{cases} \]
where we delete ``1'' in the upper index, keeping only the sign. (Eg. $a^{-} b^{+}:=a^{-1} b^{+1}$.)  In a  {\em bidirected} graph $G=(V,E)$,  $E$ consists of edges of the form $a^\alpha b^\beta$. The {\em underlying graph} is obtained by ignoring the signs, and  the {\em components} of the bidirected graph are those of the underlying graph. If the sum of  a set of edges of a bidirected graph is $0$, then we call it a {\em cycle}; if it is   inclusionwise minimal   with this property, then it is a {\em circuit}. An $(a^\alpha, b^\beta)$ {\em path} is a set of edges whose sum is  $a^\alpha b^\beta$; an $(a, b)$ {\em path} is an $(a^\alpha, b^\beta)$-path for some  $\alpha,\beta\in \{1, -1\}$. If such a path exists, we  say that $b^\beta$ or simply $b$ is {\em reachable} from $a^\alpha$ or from $a$. Fix  $x_0\in V(G)$ and define 

\smallskip
\centerline{$V^+:=V_{x_0^-}^+(G):=\{ x\in V(G): x^+\hbox { is reachable from } x_0^-$ 
\},  }

\smallskip
\centerline{$V^-:=V_{x_0^-}^-(G):=\{ x\in V(G): x^-\hbox { is reachable from } x_0^-
\}.$}
Reachability and non-reachability are characterized by these sets in a slightly more involved way than in digraphs. It is convenient to assume that $\emptyset$ is an $x_0^-x_0^+$ path, so that $x_0\in V^+$. Alternatively, one may add the loop $x_0^-x_0^+$, which does not change the set of reachable vertices; thus the assumption $x_0\in V^+$  holds throughout. We say that an edge $x^-y$ is {\em useful} at $x$, if $x\in V^+$, and  $x^+y$, if $x\in V^-$. An edge is {\em useful} if it is useful at both of its endpoints.

If $P$ is an $(a,b)$ path, then {\em  there exists a natural ordering of the edges of $P$ such that any consecutive subsequence is again a path.} Indeed, if $P$ is an $(a^\alpha,b^\beta)$ path, then it contains an edge $e=pb^\beta$, and the claim follows by induction on the number of edges, using the induction hypothesis for the path $P\setminus{e}$. It follows that every vertex of $P$ is reachable from both $a^\alpha$ and $b^\beta$.
\begin{figure}[h]
	\centering
	\includegraphics[scale=0.6]{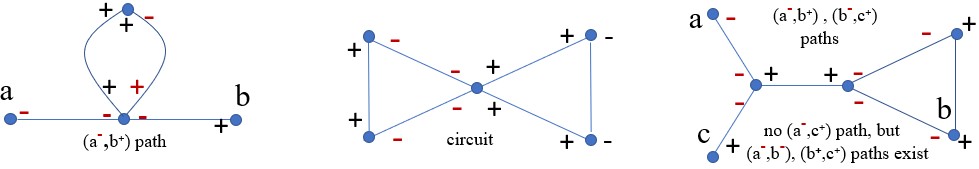}
	\caption{left: path; middle: circuit; right: failure of transitivity but rescue by Lemma~\ref{lem:trans} }
	\label{fig:bidirected}
\end{figure}
Examples of bidirected paths and circuits are shown in Figure~\ref{fig:bidirected}. Reachability in bidirected graphs is equivalent to alternating walks in two-edge-coloured graphs\footnote{Let $a^-b^-$ be red, $a^+b^+$ be blue, and $e=a^-b^+$ be replaced by $a^-v_e^-$ and $v_e^+b^+$; reachability by (possibly shortest) alternating walks (by paths as well) is polynomially solvable similarly to shortest odd paths \cite{SCHRIJVERyellow}.},  and turns out to provide the improving paths for graph factors. The bidirected notation has a simplifying role. The following weaker property replaces transitivity (see Figure~\ref{fig:bidirected}) and plays a key role.	

\vspace{-0.22cm}
\begin{lemma}\label{lem:trans}\cite{thesis} If there exist both   $(a^\alpha,b^\beta)$ and $(b^{-\beta},c^\gamma)$ paths, then in addition there exist either an $(a^\alpha,c^\gamma)$ path or  both  $(a^\alpha, b^{-\beta})$ and  $(b^\beta, c^\gamma)$ paths  (Figure~\ref{fig:bidirected}).
\end{lemma}

\prove Let $P$ be an $(a^\alpha,b^\beta)$ path, and $Q$ a $(b^{-\beta},c^\gamma)$ path. Starting from  $a^\alpha$ along $P$ let  $q$  be  the first vertex of $Q$. Exactly one of $P(a^\alpha ,q)\cup Q(q,c^\gamma )$ and  $P(a^\alpha ,q)\cup Q(q,b^{-\beta })$ is a path.\endproof

\begin{theorem}\label{thm:reachability}\cite{thesis}
	Let $G=(V,E)$ a bidirected graph, $x_0\in V$, and $V^+, V^-$ as above.  Then:
	
\vspace{-0.2cm}
\begin{itemize}
	\item [{\rm(i)}] No edge induced by $V^+\Delta V^-$ is useful.
	\vspace{-0.1cm}
	\item [{\rm(ii)}]	Let $C$ be a component of $G-(V^+\triangle V^-)$. Then 
	
	-- either $C\cap (V^+\cup V^-)=\emptyset$, and there is no edge in $\delta(C)$ useful in $V^+\cup V^-$,
	
	-- or $C\subseteq V^+\cap V^-$,  and then 
	
	\qquad if $x_0\notin C$, exactly one edge of $\delta(C)$ is useful;  
	
     \qquad if $x_0\in C$,no edge of $\delta(C)$ is useful. \vspace{-0.2cm}
\end{itemize}
\end{theorem}

\noindent{\bf Proof}:(Sketch) (i) is easy. Indeed, a useful edge  $uv$ induced by $V^+\Delta V^-$ appended or deleted from a path would contradict the definition of  $V^+$ or $V^-$ . 

Assertion (ii) is straightforward for  $C$ with $C\cap (V^+\cup V^-)=\emptyset$. If $C\cap (V^+\cup V^-)\ne\emptyset$, then $\emptyset\ne C\cap (V^+\cup V^-)=C\cap (V^+\cap V^-)$, since $C$ is a component after deleting $V^+\Delta V^-$. 

Assume therefore $C\cap V^+\cap V^-\ne\emptyset$, and prove $C\subseteq V^+\cap V^-$ together with the remaining claims; we can assume $x_0\notin C$, otherwise  rename $x_0$ to $r$, introduce a new vertex $x_0$, an edge $x_0^- r^+$, and then  apply the  argument to the modified graph. 

Let $x\in C\cap V^+\cap V^-$ and $e\in \delta(C)$ the first   edge entering $C$ of an $(x_0^-,x)$ path; $e$ is useful,  $e=p^-r^+$ say, $p\in V^+\setminus V^-$,   $r\in C$,  and let $R$ be the  $(x_0,r^+)$ part of the path, ending with $e$.

\noindent{\bf Claim}: $C_{r^-}^+ \subseteq C_{r^-}^-$. 

Indeed, let  $v\in C_{r^-}^+$, and  $P$ be the $(r^-,v^+)$ part inside $C$ of this path.
Since $v\in C\subseteq V^+\cap V^-$, there also exists an $(x_0^-,v^-)$ path, whose last edge is  $f=q^\alpha s^\beta$, $q\in V^{-\alpha}\setminus V^{\alpha}$, $s\in C$, $\alpha,\beta\in\{1,-1\}$.  Denote the  last $(s^{-\beta}, v^-)$ part in $C$ of this path by $Q$.

Applying Lemma~\ref{lem:trans}   to $P$ and $Q$ yields the claim. Indeed,  no path $(r^-,s^{-\beta})$ exists,  since together with $R$ and $f$ it would produce an $(x_0,q^\alpha)$ path, contradicting $q\in V^{-\alpha}\setminus V^{\alpha}$. Hence by Lemma~\ref{lem:trans} there exists an $(r^-,v^-)$ path, and the claim is proved.

\smallskip
By symmetry, $C_{r^-}^- \subseteq C_{r^-}^+$, so 
$C_{r^-}^+ = C_{r^-}^-$. Since the underlying graph is connected, $C_{r^-}^+ = C_{r^-}^-=C$, otherwise there exists an edge of $\delta (C_{r^-}^+)$ whose other endpoint is  in $C\subseteq V^+\cap V^-\setminus C_{r^-}^+$, contradicting the definition either of $C_{r^-}^+$ or of $C_{r^-}^-$.

Finally, knowing that $C\subseteq V^+\cap V^-$,  another useful edge $e'=p'r'\in\delta(C)$, $p'\in V^+\Delta V^-$,  $e'\ne e$ could be added to the union of $R$ and an $(r^-,-r')$ path, producing a path ending in $V^+\Delta V^-$ with  the wrong sign. This contradiction completes the proof.
\endproof

We can call $X^+, X^-\subseteq V$, $x_0\in X^+$, satisfying (i) (ii) when substituted  for  $V^+, V^-$ a {\em bidirected cut} separating $x_0^-$ from some of the vertices with some signs. Indeed,
it is easy to see that for
 $u\in V\setminus X^+$ and $v\in V\setminus X^-$, neither $u^+$ nor $v^-$ is then reachable from $x_0^-$  \cite[B.Tétel, page 77]{thesis}. Together with Theorem~\ref{thm:reachability} this extends Minty's theorem 
 to bidirected graphs.
 
  The proof can be turned to an algorithm of complexity $O(n^4)$  by keeping paths certifying membership in $V^+$ and $V^-$ and growing these with new paths until (i) and (ii) are  both satified, analogously with the algorithmic proof of Theorem~\ref{lem:shortest}. Again, this specializes to   Lovász's algorithm based ``upon the Gallai-Edmonds structure theorem'' \cite[9.3]{LPL}, and can be widely applied to degree-constrained subgraph or orientation problems by various definitions of a bidirected graph, implying Edmonds-Gallai type structure theorems and algorithms.
  
  We illustrate this method on  classical factor problems defined by the triple  $(l,u,T)$, and state the lemma that provides all the improving paths with respect to a current $F\subseteq E$: 

\[
\begin{aligned}
	E(\overleftrightarrow{G_F}) =\;& 
	\{\,x^-y^-:  xy \in F \,\}
	\;\cup\;
	\{\, x^+ y^+ : xy \in E(G)\setminus F \,\}
	\\[4pt]
	&\cup\; 
	\{\, x_0^- x^- : d_F(x) < l(x) \,\}
	\;\cup\;
	\{\, x_0^- x^+ : d_F(x) > u(x) \,\}
	\\[4pt]
	&\cup\;
	\bigcup_{x \in T,\; l(x)<d_F(x)<u(x),\; d_F(x)\not\equiv u(x)\!\!\pmod{2}}
	\{\, x_0^- x^+,\, x_0^- x^- \,\}
	\\[4pt]
	&\cup\;
	\{\, x^+ x^+ : x\in T,\; l(x)<d_F(x)\le u(x),\; d_F(x)\equiv u(x)\!\!\pmod{2} \,\}
	\\[4pt]
	&\cup\;
	\{\, x^- x^- : x\in T,\; l(x)\le d_F(x)<u(x),\; d_F(x)\equiv u(x)\!\!\pmod{2} \,\}.
\end{aligned}
\]

\begin{lemma}\label{lem:improving}\cite{thesis}
$F$ is not optimal,  if and only if at least one of the following hold:  

{\rm (i)} $x_0\in V^+\cap V^-$,

{\rm (ii)} there exists $x\in V^+ \setminus \Pi$,

{\rm (iii)} there exists $x\in V^- \setminus \Pi$, such that $d_F(x)> l(x)$.
\end{lemma}
\prove (Sketch) Conditions (i)-(iii) provide {\em improving paths} i.e. paths decreasing the deficiency; conversely, if $F'$ has smaller deficiency than $F$, then $F\Delta F'$ contains one of the paths of (i)-(iii).\endproof
Edmonds' blossoms can be adapted to find these  improving paths in $O(n^2)$ time.

\smallskip
\noindent{\bf Proof of Theorem~\ref{thm:factor}}:  The  ``easy direction'' was verified immediately after the statement. For the converse, assume that none of  (i)-(iii) holds, i.e.:

(i)  $x_0\in V^+$,\qquad (ii) $V^+\setminus \Pi\subseteq \{x:d_F(x)\ge u(x)\}$,\qquad  (iii) $V^-\setminus \Pi\subseteq \{x:d_F(x)\le l(x)\}$. 

With the choice  $L:=V^+\setminus \Pi$, $U:=V^-\setminus \Pi$  the equality holds in all  lower bounds of the easy part of the proof, so the violation of the condition by $F$ is equal to    the difference of the left and right hand sides of Theorem~\ref{thm:factor}. If $F$ is not an $(l,u,\Pi)$-factor, this difference is positive, finishing the proof.  A  min-max theorem for the deficiency + a structure theorem  can also be read out. 
 \endproof

Orientations subject to upper, lower bounds and possibly parity constraints on say the indegrees, can be treated in a way  similar to those for graph factors, and the results are also similar \cite{FST}. Given a non-empty set $H\subseteq\mathbb{Z}_+$, an orientation of $G$ will be said to be {\em $H$-respecting}, if the indegree of each $v\in V$ lies in $H(v)$. Here is an explanation for the similarities.

Theorem~\ref*{thm:reachability} applies to $H$-respecting orientations in a similar way to the above proof of Theorem~\ref{thm:factor}, including both the optimisation (minimisation of the deficiency (distance from $H$), and algorithmically, also yielding the corresponding Edmonds-Gallai theorems. (For classical sponges this is worked out   in  \cite{thesis}.) The extension from subgraphs to orientations follows in fact from a simple reduction: 
subvivide each $e\in E$ by  a new vertex $v_e$, set $\hat H(v_e):=\{1\}$, and keep $H(v)$ unchanged for the original vertices. The new vertices are classical. 
There is a natural bijection between $\hat H$-factors $M$ of the subdivided graph and  $H$-respecting orientations of the original  graph (also preserving deficiency): orient $e=xy\in E$ towards $y$ if  and only if
$v_ey\in M$.
Finally, this framework also applies to mixed graphs, in which edges may be treated as factor edges, as oriented edges, or left free to be assigned either role. The corresponding results are not stated separately, as they follow  automatically from the corresponding factor theorems; the reader interested in orientations can bear this in mind throughout the next section.

\section{Jump systems intersecting sponges}\label{sec:jump} 

It remains to put the  icing on the cake: the missing and promised  polynomial algorithm for sponge deficiency.  Lov\'asz \cite{factors}  developed the necessary structural results that he generalized later to jump systems in  \cite{Ljump}. As already noted (see the introductory part of Section~\ref{sec:classical}),  it is natural to restrict attention to sponges: allowing gaps of size two the problem becomes \NP-hard.

Graph factors can be understood as intersections of degree vectors of subgraphs and sponges. Both are ``jump systems''. This seems to be essential for their tractability. 

Other examples of jump systems are independent sets or bases of matroids, feasible sets of delta-matroids,  integer points of polymatroids or bisubmodular polyhedra. We  first describe an algorithm  to decide in polynomial time whether the intersection of a degree vector jump system with a sponge is   non-empty (Section~\ref*{subsec:general}). We then extend the argument to general jump systems, showing that essentially the same ideas apply (Section~\ref{subsec:jump}). Further refinements are possible e.g. trading restrictions on one of the jump systems  for greater generality of the other. 

Jump systems were  introduced by Bouchet and Cunningham \cite{BC} as a common  generalisation of degree vectors of subgraphs and delta-matroids. These, in turn, extend  matchable sets and matroids respectively.  Although jump systems include integer polymatroids, they are best viewed as genuinely combinatorial rather than polyhedral objects: like degree vectors, they may have “holes”, i.e.~integer points in their convex hull that do not belong to them.

\smallskip
Let $x\in\mathbb{Z}^V$. A {\em step at $v$ $(v\in \mathbb{Z})$ from}  $x$  is  an $x'\in\mathbb{Z}^V$ such that $x-x'=\pm e_v$ $(v\in V)$; we say that it is a step {\em towards}  a set $S\subseteq \mathbb{Z}^V$, if   $\mu(x', S)<\mu(x, S)$. Then clearly, 
$\mu(x', S)=\mu(x, S)-1$. 

\smallskip
 A {\em jump system} is a set $J\subseteq \mathbb{Z}^n$,  if $J\ne\emptyset$ and  for any $x, y\in J$ and  first step $x'$ from $x$ to $y$, either $x'\in J$, or there exists a step $x''$ from $x'$ towards $y$ such that 
$x''\in J$. 

In order to avoid case separations between one or two steps, in case of $x'\in J$ we  define the second step to be $x''=x'$.  (Both $\mu(x'', y)=\mu(x, y)-1$ (if $x''=x'$), and $\mu(x'', y)=\mu(x', y)-1=\mu(x, y)-2$ are possible.)  The definition of jump systems is called the {\em $2$-step-axiom}.  

Besides the mentioned examples encompassed by jump systems, one dimensional jump systems present another relevant example,  they are exactly the sponges:
\vspace{-0.15cm}
\begin{fact}\label{fact:onedim} Let $H\subseteq \mathbb{Z}$, $H\ne\emptyset$.  The following statements are equivalent:
	
{\rm(i)}  $H$ is a jump system,

{\rm(ii)}  $H$ is a sponge,

{\rm(iii)} $H$ has no gap of length two in the interval $[l,u]$, where $l:=\min H$, $u:=\max H$. \endproof
\end{fact}

\noindent{\bf Problem}: (deficiency (of jump systems))  Let $J$, $H$ be jump systems. Is $J\cap H\ne\emptyset$? More generally, compute $\mu(J,H)$! Equivalently, find $z\in J$ minimizing $\mu(z,H)$. 

\medskip
The latter, asymmetric formulation is appropriate when $H$ is substantially simpler than $J$, e.g. when it is  given explicitly; this is the case when $H$ is a sponge.  (The general deficiency problem  is in fact equivalent to the problem of minimizing  $\mu(0, J-H)$, which is a jump system membership problem, and in this generality it would be hard to say  about it more than  \cite{Ljump}.) We will therefore seek for an {\em optimal} $z$ in $J$, i.e. $z\in J$ minimizing $\mu(z,H)$ for a given sponge $H$.  

Depending on whether $H$ in the above Problem is a general  sponge, or it is restricted to be classical, interval or parity, we call it a {\em sponge, classical, interval or parity} deficiency problem. 

The {\em general factor problem} is the sponge deficiency problem where $J$ is the set of degree vectors  of graphs.

For any of these problems to make sense it has to be clear how a jump system is given. The literature abounds in choices. The reason is that there is no unique best assumption for an oracle, the choice is the result of a nuanced balance between the specific goals and convenience. An overly strong oracle may kill the applications, an overly weak one prevents efficient work. 

Our modest ambitions conveniently resolve this difficult choice: the classical deficiency problem has been explored in  \cite{Ljump}, and admits good characterisations and polynomial-time optimisation algorithms under general assumptions, and in all special cases we know. (E.g.   $O(n^2)$ for degree vectors; another  example occurs in Section~\ref{subsec:jump}). We therefore adopt it as our oracle.

Our objective is thus to find an optimal $z\in J$ for a given sponge $H$, {\em assuming} that optimisation over $J$ is solved for any classical sponge; in fact an oracle for parity sponges suffices, and  this generality can be exploited  (e.g. Corollary~\ref{cor:classicalvsparity}). For a class of jump systems $\Jscr$  a {\em parity oracle} computes, in a single step, the deficiency of a jump system $J\in \Jscr$, $J\subseteq\mathbb{Z}^V$  with respect to any parity sponge $H\subseteq\mathbb{Z}^V$, and finds an optimal $z\in J$.   This choice isolates the genuinely new  challenges for sponges, relegating the classical deficiency problem  to existing theory. 
 
 We prove in the following two subsections:
 
 \noindent
 {\em An optimal general factor can be  found by solving $(2n+1)n$ parity factor problems;  the same number of parity oracle calls and a negligible number of other operations are sufficient for finding an optimal element of a jump system with respect to a sponge.}
  
To prove this we need to define  the {\em environment} of a point $z\in\mathbb{Z}^V$ in a sponge $H\subseteq \mathbb{Z}^V$. Intuitively, $H_z \subset H$ is an inclusionwise maximal parity sponge included in $H$ ``closest'' to $z$. For any jump system   $J\subseteq \mathbb{Z}^V$ and sponge $H\subseteq\mathbb{Z}^V$, it will turn out that $z\in J$ is   not optimal   for $H$ if and only if  either it is not optimal for $H_z$,  or it can be {\em improved by  escaping} from $H_z$. 

  We now define $H_z$ precisely (Figure~\ref{fig:environment}). 
Let $H\subseteq \mathbb{Z}$ be a sponge and $z\in  \mathbb{Z}$  arbitrary. In $H$ the (parity-){\em environment} $H_z$ of $z$  is the intersection $H\cap [l_z,u_z]$, where   
\[l_z:=\min\{h\in H: \hbox{ every integer in $H\cap [h, z]$ has the same parity}\}, \]
\[u_z:=\max\{h\in H: \hbox{ every integer in $H\cap [z, h]$ has the same parity}\}.\]	
There may be $0$, $1$, or $2$ {\em escape values}  from  $H_z$, 
and these are $l_z-1$ if $l_z-1\in H$, and $u_z+1$ if $u_z+1\in H$, respectively. Escape values mark changes in the environment. 
\begin{figure}[h]
	\centering
	\includegraphics[scale=0.3]{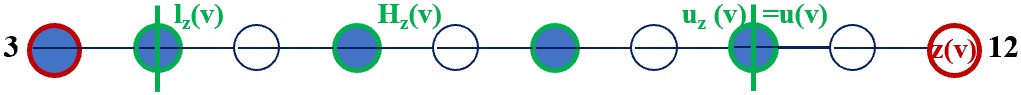}
	\caption{Environment of $z(v)=12$ with one escape ($=3$); if $d_F(v)=12$, $H_F=\{4,6,8,10\}$. }
	\label{fig:environment}
\end{figure} 	

\vspace{-0.25cm}
For  $H\subseteq \mathbb{Z}^V$ and $z\in  \mathbb{Z}^V$ we define $H_z\subseteq \mathbb{Z}^V$ entrywise, i.e.~$H_z$ is the Cartesian product of the sets  $H_{z(v)}(v)$ $(v\in V)$. 
For each $v\in V$ we define at most two possible {\em escapes} (from $H_z$): $H_{z,\underline v}, H_{z,\overline v}\subseteq \mathbb{Z}^V$. These are obtained with the same Cartesian product as $H_z$, but the component $H_{z(v)}(v)$ is replaced by a singleton containing the corresponding escape value. Specifically, the lower escape uses  $\{l_z(v)-1\}$ provided that $l_z(v)-1\in H$, and the upper escape uses $\{u_z(v)+1\}$ provided that $u_z+1\in H$.  If the relevant condition is  not satisfied the corresponding escape is undefined. Consequently there are at most $2n$ escapes from $H_z$.

\begin{fact}\label{fact:environment}  Let $H\subseteq \mathbb{Z}^V$ be a sponge. Then
	
	{\rm (i)} for all $z\in\mathbb{Z}^V$ : $H_z\subseteq \mathbb{Z}^V$ is a parity sponge, and   $\mu(z,H_z)=\mu(z,H)$; 
	
		{\rm (ii)} if for some $z\in\mathbb{Z}^V$ there is no lower escape, i.e.~$l_z(v)-1\notin H(v)$,  then $l_z(v)=l(v)$; 
	
		{\rm (iii)} similarly, $u_z(v)+1\notin H(v)$ implies $u_z(v)=u(v)$.\endproof
\end{fact}

The environment and the $2n$ escapes -- at most $2n+1$  sets of vectors -- were introduced in  \cite{thesis} for $z:=d_F$, where $F\subseteq E$, and then $H_z$, $H_{z,\underline v}$,  $H_{z,\overline v}$,  $l_z$, $u_z$ are denoted  $H_F$, $H_{F,\underline v}$, $H_{F,\overline v}$, $l_F$, $u_F$ respectively. They provide a helpful framework for solving the general factor problem via classical deficiencies (Section~\ref{subsec:general}), and then for extending this solution to jump systems  (Section~\ref{subsec:jump}).  

\subsection{Degree vectors and sponges: general factors}\label{subsec:general}

Lovász's article on {\em antifactors} \cite{antifactors}, and his question  on factors satisfying sponge constraints \cite{factors, LPL}, first motivated a generalisation to  $H$-factors,  where $H$ is a sponge without upper and lower  bounds \cite{AS}, i.e. with  $l(v)\in\{0,1\}$, $u(v)\in\{d(v)-1, d(v)\}$, and an arbitrary number of gaps of length one. 

The general factor problem with possible upper and lower bounds, i.e. for sponges in full generality, was solved by Cornuéjols \cite{corn1}. His approach is based  on an Edmonds-type algorithm leading to a Tutte-style characterisation \cite[Section 3]{corn1}, \cite[Section 4]{corn2}. For cubic graphs, the problem reduces to edge–triangle partitioning \cite[Section 2]{edgetriangle,corn1,corn2}.  

While powerful enough for  solving the general factor problem, this method is technically involved and the  certificate of infeasibility replaces odd parity of “critically non-matchable components” with an  execution of the algorithm itself rather than with a transparent   parity condition.   

A shorter approach was thereafter  proposed in  \cite[Section~4.4]{thesis}, giving a  concise reduction of the {\em deficiency of general factors}  to classical, and in fact purely parity deficiency problems; an account  was outlined in  \cite[Section 3]{corn2}. We reproduce here the full solution of  \cite{thesis}, in a form suited for extension to jump systems in the next subsection.

This solution relies on classical algorithms as a black box and yields no structural characterisation. It nevertheless has several advantages: in addition to deciding feasibility, it also minimizes the deficiency, leads to improved complexity bounds, and — most importantly for us here — is readily adaptable to the deficiency problem for jump systems: Section~\ref{subsec:jump} on this extension to jump systems will closely  mimic the graph special case of the present section. 

Our approach is based on  a simple notion introduced by Lovász \cite{factors}:  
given a graph $G=(V,E)$, a sponge $H$, and $F\subseteq E$, an {\em elementary change}  (with respect to $H$) is the removal or addition of an edge $e=ab\in E$  so that distance of $d_F(a)$ from  $H(a)$ is nonzero and does not increase with the change; we denote this elementary change by $(F;a,b)$.
 
 In other words, the elementary change $(F;a,b)$ supposes that $ab\in E$, and
\begin{equation}\label{eq:echange}
	0\ne\mu(d_{F}(a),H(a))\ge \mu(d_{F\Delta e}(a),H(a)). 
	 \vspace{-0.2cm}
\end{equation}

\begin{fact}\label{fact:elementary}
 Let $H$ be a sponge,  $e=ab\in E$, and let $(F;a,b)$ be an elementary change. Then:
\begin{itemize}
\vspace{-0.1cm}
\item[{\rm (i)}] $\mu(d_{F\Delta \{e\}}(a), H(a)) < \mu(d_F(a), H(a))$, unless $e$ is a loop with 
the number between $d_{F}(a)$ and $d_{F\Delta \{e\}}(a)$ in $H(a)$,  $\mu(d_{F\Delta \{e\}}(a), H(a)) = \mu(d_F(a), H(a))=1$,
\vspace{-0.1cm}
\item[{\rm (ii)}]  $\mu(F\Delta \{e\}, H)\le  \mu(F, H)$,
\vspace{-0.1cm}
\item[{\rm (iii)}] if $H_{F\Delta \{e\}}\ne H_F$, then $d_{F}(b), d_{F\Delta \{e\}}(b)\in H(b)$, $\mu(F\Delta \{e\}, H)<\mu(F,H)$, furthermore

\,\,\,\quad -- either $e\in F$, $d_{F\Delta \{e\}}(b)=l_F(b)-1 \in H(b)$,\, and  $\mu(F\Delta \{e\}, H_{F,\underline b})<\mu(F,H)$;

\,\,\,\quad --  or\,\,\,\,\,\,\,\,\, $e\notin F$,   $d_{F\Delta \{e\}}(b)=u_F(b)+1 \in H(b)$, and 
$\mu(F\Delta \{e\}, H_{F,\overline b})<\mu(F,H)$.
\end{itemize} 
\end{fact}

\prove (i) can be readily checked from (\ref{eq:echange}) using also that $H$ is a sponge; (i) implies (ii). 

For (iii), suppose for a contradiction that  $d_{F\Delta \{e\}}(b)\notin H(b)$. Then, since $H$ is a sponge,  \eqref{eq:bounds_unchanged}  and consequently (\ref{eq:unchanged}) hold for $v=b$; by  (i), \eqref{eq:bounds_unchanged}  and (\ref{eq:unchanged}) hold for $v=a$ as well.  
\begin{equation}
	l_{F\Delta \{e\}}(v)=l_F(v),\, u_{F\Delta \{e\}}(v)=u_F(v),  \vspace{-0.2cm}
	\label{eq:bounds_unchanged}
\end{equation}  
\begin{equation}
	H_{F\Delta \{e\}}(v)=H_F(v).  \vspace{-0.1cm}
	\label{eq:unchanged}
\end{equation}

For $v\in V\setminus\{a,b\}$ \eqref{eq:bounds_unchanged},  \eqref{eq:unchanged} 
are   obvious. Thus $H_{F\Delta \{e\}}=H_F$,
contradicting the assumption, and proving $d_{F\Delta \{e\}}(b)\in H(b)$, moreover: 
\vspace{-0.3cm}
\begin{equation}
	d_{F\Delta \{e\}}(b)\in H(b),\quad d_{F}(b)\in H(b),\quad l_F(b)=d_F(b), 
 \vspace{-0.1cm}
	\label{eq:consec}
\end{equation}
because  $d_{F}(b)\notin H(b)$  would imply  \eqref{eq:bounds_unchanged} and then  \eqref{eq:unchanged} would again hold for all $v\in V$. 

Let us check next that under the condition of (iii) we have:
\vspace{-0.2cm}
\begin{equation}
	\mu(F\Delta \{e\}, H)<\mu(F,H).  \vspace{-0.2cm}
	\label{eq:improved}
\end{equation} 
The exception of (i) cannot occur here: it would imply $a=b$ and hence \eqref{eq:bounds_unchanged}, \eqref{eq:unchanged}  for all $v\in V$. Therefore,  $\mu(d_{F\Delta \{e\}}(a), H(a)) < \mu(d_F(a), H(a))$ by (i), while
$\mu(d_{F\Delta \{e\}}(b), H(b)) \le \mu(d_F(b), H(b))$ by the already proven $d_{F\Delta \{e\}}(b)\in H(b)$. We obtain   \eqref{eq:improved} adding up the strict inequality for $a$, the inequality for $b$,  and $\mu(d_{F\Delta \{e\}}(v), H(v))=\mu(d_F(v),H(v))$ for $v\in V\setminus\{a,b\}$.  

Finally, we distinguish cases:

If $e\in F$, then $d_{F\Delta \{e\}}(b)=d_F(b)-1$, so $d_F(b), d_F(b)-1\in H(b)$, $l_F(b)=d_F(b)$, and
\vspace{-0.2cm}
\begin{equation}
	\mu(F\Delta \{e\}, H)=\mu(F\Delta \{e\}, H_{F\Delta \{e\}})=\mu(F\Delta \{e\},H_{F,\underline b}), \vspace{-0.2cm}
	\label{eq:eq}
\end{equation}
where the first equality is  Fact~\ref{fact:environment}(i); we see the second equality by observing $H_{F\Delta \{e\}}(v)=H_{F,\underline b}(v)$ for all $v\in V$, using \eqref{eq:consec}. 

The case  $e\notin F$ follows by symmetry, finishing the proof of   Fact~\ref{fact:elementary}. \endproof

\smallskip\noindent
{\bf Note}: The roles of $a$ and $b$ are {\em not} symmetric in  $(F;a,b)$: (i) holds for $a$, but replacing $a$ by $b$ we may have the reverse strict inequality; if we don't, we can improve; if $H_{F\Delta \{e\}}=H_F$ we improve with respect to $H_F$ and thus with respect to $H$ by Fact~\ref{fact:environment}(i); if $H_{F\Delta \{e\}}\ne H_F$ we improve  by Fact~\ref{fact:elementary} (iii). This argument will be formalized in the proof of Theorem~\ref{thm:sponge_fact} below.

 Both cases of (iii) hold if $u_F=l_F$, i.e. if $u_F-1, u_F, u_F+1\in H$.  

\smallskip
For $F,F'\subseteq E$ write $F\rightarrow_H F'$ if $F'$ can be obtained from $F$ by a succession of elementary changes (with respect to $H$).  Mostly,  $H$ is fixed and clear from the context; then we delete it in the index. Clearly, the relation ``$\rightarrow$'' is transitive. 

The following lemma and  proof are from  \cite{factors}. We include the proof  to highlight where the sponge property of $H$ is used (Case 3), preparing the extension to jump systems in Section~\ref{subsec:jump}.

\vspace{-0.2cm}
\begin{lemma}\label{lem:LL}
	Let $G=(V,E)$ be a graph and let $H\subseteq\mathbb{Z}_+^V$  be a sponge.
	Then for every $F\subseteq E$ there exists an $H$-optimal 
	$F_H\subseteq E$ such that $F\rightarrow F_H$.
\end{lemma}
\vspace{-0.2cm}

\prove
	Let $F\subseteq E$ be arbitrary, $F^*\subseteq E$ be $H$-optimal,
	and let $F_H\subseteq E$ satisfy $F\rightarrow F_H$.
	Choose $F^*$ and $F_H$ so that $|F^*\Delta F_H|$ is minimum.
	We claim that $F^*=F_H$.
	
	Suppose otherwise. Then $F_H$ is not optimal, so there exists
	$a\in V$ such that
		\vspace{-0.2cm}
	\begin{equation}\label{eq:nopt}
		\mu\bigl(d_{F^*}(a),H(a)\bigr)
		<
		\mu\bigl(d_{F_H}(a),H(a)\bigr).
		\vspace{-0.3cm}
	\end{equation}
	
	\medskip
	\noindent\textbf{Claim.}
	There exists an edge $e=ab\in F^*\Delta F_H$ such that
	$(F_H;a,b)$ is an elementary change.
	
	\smallskip
	\noindent\emph{Case 1.} $d_{F_H}(a)\ge u(a)$.
	Then~\eqref{eq:nopt} implies $d_{F^*}(a)<d_{F_H}(a)$.
	Choose any edge $e=ab\in F_H\setminus F^*$.
	
	\smallskip
	\noindent\emph{Case 2.} $d_{F_H}(a)\le l(a)$.
	Then similarly,  $d_{F_H}(a)<d_{F^*}(a)$.
	Choose any edge $e=ab\in F^*\setminus F_H$.
	
	\smallskip
	\noindent\emph{Case 3.} $l(a)< d_{F_H}(a)< u(a)$.
	Then~\eqref{eq:nopt} yields $d_{F_H}(a)\notin H(a)$.
	Since $H$ is a sponge, we have
	$d_{F_H}(a)-1,d_{F_H}(a)+1\in H(a)$.
	Choose any edge incident to $a$,  $e=ab\in F^*\Delta F_H$.
	
	\smallskip
	In all cases, the chosen edge $e$ satisfies the claim.
	
By the claim, $F\rightarrow F_H':=F_H\Delta\{e\}$, and
		$|F^*\Delta F_H'|<|F^*\Delta F_H|$,
contradicting the minimality of $|F^*\Delta F_H|$.
	Hence $F^*=F_H$, completing the proof.
\endproof

We now state and prove the promised reduction of general factors to classical factors \cite{thesis}: 
 \begin{theorem}\label{thm:sponge_fact}
 	Let $G=(V,E)$ be a graph,  $F\subseteq E$, and $H\subseteq\mathbb{Z}^V$ a sponge. Then $F$ is not  optimal  if and only if there exists an edge-set     $Y\subseteq E$ with smaller distance  than $\mu(F,H)$ from $H_F$ or from at least one of the $2n$ escapes.
 \end{theorem}
It is important to note a subtle point in this assertion: the goal is not to decrease the deficiency of $F$ with respect to the environment or the escapes, but to decrease one of these below $\mu(F,H)$.  
 In other words, the optimality of $F$ with respect to $H$   cannot be simplified to  merely the optimality with respect to  the  $2n+1$ parity sponges concerned.
 
  Consider a graph on two vertices, $a$ and $b$ joined by  $9$ parallel edges. Let 
 \vspace{-0.25cm}
  \[H(a):=\{0,1,3,5,7,9\},\quad H(b):=\{2,4,6,8\},\quad \hbox{and}\quad F:=E.\vspace{-0.25cm}\]
   
  Then $F$ is optimal, and  $\mu(d_F,H)=1$. But $F$ is not optimal with respect to 
$H_{F, \underline a}$: since  $d_F(a)=9$, $l_F(a):=l_9(a)=1$, so $l_F(a)-1=0$, and 
$H_{F, \underline a}(a)=\{0\}$,  $H_{F, \underline a}(b)=\{2,4,6,8\}$. Hence
$\mu(F,H_{F, \underline a})=10$, whereas for $F'$ consisting of a single edge, $\mu(F', H_{F, \underline a})=1$. Thus, although $F$ is not optimal with respect to $H_{F, \underline a}$,   Theorem~\ref{thm:sponge_fact} does establish optimality: the minimum deficiency of $F$ with respect to $H_F$ or  to the escapes is not smaller than $\mu(F,H)=1$.

{\em Improving with respect to one of the escapes does not necessarily imply an improvement with respect to $H$.} (Improving with respect to the environment obviously does by Fact~\ref{fact:environment}(i).)

If parallel edges are to be avoided, one may subdivide each edge $e$ into two edges with a new vertex $v_e$, and assign $H(v_e):=\{0,2\}$ on the new vertices $v_e$ $(e\in E)$, obtaining an equivalent example without parallel edges.

\smallskip\noindent
{\bf Proof  of Theorem~\ref{thm:sponge_fact}}: The if part is easy: since $H_F, H_{F, \underline v}, H_{F, \overline v}\subseteq H$, for all $v\in V$, we have 

\smallskip
\centerline{\hbox{for each $F'\subseteq E$ }, $\displaystyle\mu(F',H)\le \min\big\{\mu(F',H_F),  \cup_{v\in V}\{\mu(F',H_{F, \underline v}),  \mu(F',H_{F, \overline v})\} \big\},$ }

\smallskip\noindent
and therefore, if the right hand side is smaller than  $\mu(F,H)$,  then so is the left hand side, showing that $F$ is not optimal. 

To prove the only if part, suppose that $F$ is not optimal,   and apply   Lemma~\ref{lem:LL}: let 
\vspace{-0.25cm}
\begin{equation}\label{eq:notopt}
F_H\subseteq E,\,  F\rightarrow F_H,\, \mu(F_H, H)<\mu(F,H).\vspace{-0.25cm}
\end{equation} 
Let $\Fscr$ be the family of edge-sets occurring in the series of elementary changes starting with $F$ and ending with $F_H$.
 
If for all $\Phi\in\Fscr$ we have $H_{\Phi}=H_F$,  we are done, 
for then $H_{F_H}=H_F$, and 
\vspace{-0.25cm}
\[\mu(F_H,H_{F})= \mu(F_H,H_{F_H})= \mu(F_H,H)<\mu(F,H),\vspace{-0.25cm}\]
where the first equality repeats $H_{F_H}=H_F$,  the second is  Fact~\ref{fact:environment}(i) applied to $z:=F_H$, and the last,  strict inequality is (\ref{eq:notopt}).

Otherwise, let $Y\in \Fscr$, $H_{Y}= H_F$ such that the next edge-set in the trajectory $F\rightarrow Y$ is  $Y\Delta e$,  where $Y\Delta e$ is the first set of the trajectory with $H_{Y\Delta e}\ne H_F$, $(e=ab\in E)$. 

Suppose first that $e\in Y$. Since  $H_{F}= H_{Y}$,  we have  $H_{F, \underline b}=
H_{Y, \underline b}$. Applying Fact~\ref{fact:elementary}~(iii) to $Y$, $e$ and the sponge $H_{Y, \underline b}$, and then Fact~\ref{fact:elementary}~(ii) successively along the  elementary changes through each $\Phi\in\Fscr$ in the trajectory $F\rightarrow Y$, we obtain
\[\mu(Y\Delta e,H_{F, \underline b})=\mu(Y\Delta e,H_{Y, \underline b})< \mu(Y,H)\le\ldots\le \mu(\Phi,H) \le \ldots \le \mu(F,H),\]
i.e. the distance of $Y\Delta e$  from the escape $H_{F, \underline b}$ is strictly smaller than $\mu(F, H)$. If   $e\notin Y$, in an entirely symmetric way, the distance of $Y\Delta e$  from the escape $H_{F, \overline b}$ is smaller than $\mu(F, H)$. 
\endproof

\begin{corollary}\label{cor:complexity}
An optimal general factor is  found by solving $(2n+1)n$ parity factor problems.   
\end{corollary}
\prove For given $F$, optimizing for $H_F$ and the $2n$ escapes takes $2n+1$ parity optimisation problems. If none of these improves $\mu(F,H)$,  then by Theorem~\ref{thm:sponge_fact} $F$ is optimal. Otherwise, again by the theorem,  $Y\subseteq E$, $\mu(Y,H)<\mu(F,H)$ is found.

Starting with $F$ minimizing the deficiency of the parity factor problem $(l,u,V)$, $\mu(F,H)$ is at distance at most $n$ from the optimum for $H$, since $H$ is a sponge.\endproof

The subtle point noted after Theorem~\ref{thm:sponge_fact} deserves attention when bounding the complexity. It is stated in  \cite[Section~3]{corn2} that “any improved solution relative to” [the environment and its escapes as we defined it] “is also an improved solution relative to $H$”.  As the example above the proof shows, this is unfortunately not true in general, our task is more complex: 

{\em We have to improve the distance from the environment or escapes below $\mu(F,H)$.} 

Using the  rough bound of $O(n^4)$ mentioned above for finding a first $F\subseteq E$ and that  $\mu(F,H)$ is then at distance at most $n$ of the optimum  (cf.~Proof of Corollary~\ref{cor:complexity}), we see that  finding improving paths  in $O(n^2)$ time (cf.~Lemma~\ref{lem:improving} and thereafter)   for parity factors, and  at most $(2n+1)n + n$ times, results in an overall complexity of $O(n^4)$.   

In particular, the following curiosity follows: 

\begin{corollary}\label{cor:classicalvsparity}
A classical deficiency problem for degree vectors can be solved by solving $(2n+1)n$ parity factor problems, in $O(n^4)$ time.\endproof
\end{corollary}
Dudycz and Paluch \cite{DudyczPaluch} consider the problem of optimizing  the weight of a general factor (minimisation is equivalent to maximisation).  This is a difficult problem for which they ultimately obtain a relatively  simple direct algorithm, albeit with a technically involved proof. Kobayashi \cite{kobayashi} generalized the cardinality case   to ``strongly base orderable'' jump systems and provided a simpler proof. Could these results be also proved by reducing them to weighted classical factors?

\subsection{Jump systems and sponges}\label{subsec:jump}

A workshop held in Grenoble in 1995 on ``Delta-matroids, jump systems and bisubmodular polyhedra'' concluded with a lecture of Lovász presenting   results  developed during the meeting.  These  extended his theorems on classical graph factors  to  jump systems \cite{Ljump}. 

As in the case of degree vectors \cite{factors, LPL},  however,  \cite{Ljump} leaves open algorithmic questions about  the deficiency problem of jump systems with respect to $H$,  when $H$ is a general sponge. This issue is  explicitly raised in the closing paragraph of  \cite{LPled}.

The intersection of jump systems and sponges had in fact already been addressed in  \cite{Sjump}, where a connection with a problem of Recski also became apparent. 
 The   tools offered by  \cite{Ljump} were instrumental:
methods known for classical factors were translated into the language of jump systems.  
For jump systems, elementary changes are replaced by {\em pushes} defined as follows, allowing one to proceed  in the sequel by full analogy with Section~\ref{subsec:general}.

Let $J\subseteq \mathbb{Z}^n$ be a jump system, and $H\subseteq \mathbb{Z}^n$ a sponge. 
A {\em push at} $i\in [n]$ of $x\in J$ towards $H$  consists of a step   $x'$ at $i$ towards $H$, followed by an arbitrary step $x''\in J$ from $x'$; as usual, we set $x'':=x'$ if $x'\in J$. In contrast with the $2$-step-axiom, when  $x'\notin J$ we do not require $\mu(x'',H)<\mu(x',H)$;  only the first step must decrease the distance to $H$. We denote this push by $(x,x',x'')$. If  $x, y\in J$, a push from $x$ towards $y$ exists by the $2$-step-axiom. 
 
Fact~\ref{fact:onedim} and Fact~\ref{fact:environment} were directly stated in terms of vectors so that they could serve factors   in both graphs and jump systems. However, Fact~\ref{fact:elementary} was stated in terms of elementary changes in graphs, and we restate it for pushes because of minor differences. In jump systems both steps of a   push  $(x, x', x'')$ may occur in the same direction $i\in [n]$, (i.e.~$x'-x=x''-x'=e_i$), corresponding  exactly to loops in graphs, but here this is no longer an ``exception''; additonally,  $x'=x''\in J$ can occur, unlike in graphs. Despite these differences the assertions and arguments are mostly word for word the same, contain the graph special case, and are sometimes even simpler for vectors: 

\begin{fact}\label{fact:jelementary}
	Let $H$ be a sponge,  and $(x,x',x'')$ a push at $i$, with second step at $j$, $i,j\in [n]$. Then
	
	{\rm (i)} $\mu(x'(i), H(i)) = \mu(x(i), H(i)) - 1,$
	
	{\rm (ii)}  $\mu(x'', H)\le  \mu(x, H)$,
	
	{\rm (iii)}  if $H_{x''}\ne H_x$, then   $x(j), x''(j)\in H(j)$,  $\mu(x'',H)<\mu(x, H)$,  furthermore
	
	\,\,\,\quad -- either $x''(j)=x(j)-1$,   $x''(j)=l_x(j)-1 \in H(j)$,\, and   $\mu(x'', H_{x,\underline j})<\mu(x,H)$,
	
	\,\,\,\quad -- or \,\,\quad $x'(j)=x(j)+1$,   $x''(j)=u_x(j)+1 \in H(j)$,\, and  $\mu(x'', H_{x,\overline j})<\mu(x,H)$.
\end{fact} 

\prove (i) can be readily checked from the definition of a push, and (ii) is a consequence of (i), because $|x'(j)-x''(j)|\le 1$. 
  
For (iii), suppose for a contradiction that
$x''(j)\notin H(j)$. Then, since $H$ is a sponge,   \eqref{eq:jbounds_unchanged}  and consequently (\ref{eq:junchanged}) hold for $v=j$,  and  it can be seen from  (i) that for $v=i$ as well: 
\vspace{-0.2cm} 
\begin{equation}
	l_{x''}(v)=l_x(v),\, u_{x''}(v)=u_x(v),  \vspace{-0.2cm}
	\label{eq:jbounds_unchanged}
\end{equation}
\vspace{-0.4cm}
\begin{equation}
	H_{x''}(v)=H_x(v).  \vspace{-0.1cm}
	\label{eq:junchanged}
\end{equation}
For $v\in V\setminus\{i,j\}$ \eqref{eq:jbounds_unchanged},  (\ref{eq:junchanged}) 
are   obvious. Thus $H_{x''}=H_x$,
contradicting the assumption, and proving 
 $x''(j)\in H(j)$, moreover:
  \vspace{-0.3cm}
 \begin{equation}
 	x''(j)\in H(j),\quad x(j)\in H(j),\quad l_x(j)=x(j), 
 	\vspace{-0.2cm}
 	\label{eq:jconsec}
 \end{equation}
because $x(j)\notin H(j)$ would imply  \eqref{eq:jbounds_unchanged} and then  \eqref{eq:junchanged} would again hold for all $v\in V$.  
 
  Let us check next that under the condition of (iii) we have: 
 \vspace{-0.2cm}
 \begin{equation}
\mu(x'', H)<\mu(x,H).   \vspace{-0.2cm}
 	\label{eq:jimproved}
 \end{equation} 

We have 
 $\mu(x''(i), H(i)) < \mu(x(i), H(i))$ by (i), while
$\mu(x''(j), H(j)) \le \mu(x(j), H(j))$ by the already proven $x''(j)\in H(j)$. We obtain   \eqref{eq:jimproved} adding up the strict inequality for $i$, the inequality for $j$,  and $\mu(x'', H(v))=\mu(x,H(v))$ for $v\in V\setminus\{i,j\}$.   

Finally, distinguish cases. 

If $x''(j)=x(j)-1$, then $x''(j)=x(j)-1$, so $x(j), x(j)-1\in H(j)$, $l_x(j)=x(j)$, and
\vspace{-0.2cm}
\begin{equation}
	\mu(x'', H)=\mu(x'', H_{x''})=\mu(x'',H_{x,\underline j}), \vspace{-0.2cm}
	\label{eq:jeq}
\end{equation}
where the first equality is Fact~\ref{fact:environment}(i); we see the second equality by observing $H_{x''}(v)=H_{x,\underline j}(v)$ for all $v\in V$, using \eqref{eq:jconsec}. 

The case   $x''(j)=x(j)+1$ follows by symmetry, finishing the proof of Fact~\ref{fact:jelementary}.

\medskip For $x,y\in J$ we write  $x\rightarrow_H y$ if $y$ is obtained from $x$ by a succession of pushes. Again,  $H$ can  be deleted in the index if there is no ambiguity; ``$\rightarrow$'' is transitive. 

The following lemma and theorem extend  Lemma~\ref{lem:LL} and Theorem~\ref{thm:sponge_fact} to jump systems, with  essentially identical proofs, yet are  worth repeating for subtle novelties in the generalisation.

\begin{lemma}\label{lem:push}
	Let $J\subseteq\mathbb{Z}^V$ be a jump system, $x\in J$ and let $H\subseteq\mathbb{Z}^V$
	be a sponge.
	Then there exists an $H$-optimal   $x_H\in J$ such that
	$x\rightarrow x_H$.
\end{lemma}

\prove
Let $x\in J$ be arbitrary, $x^*\in J$ be $H$-optimal,
and let $x_H\in J$ satisfy $x\rightarrow x_H$.
Choose $x^*$ and $x_H$ so that $\mu(x^*,x_H)$ is minimum.
We claim that $x^*=x_H$.

Suppose otherwise. Then $x_H$ is not optimal, so there exists
$i\in V$ such that
\begin{equation}\label{eq:jnopt}
	\mu\bigl(x^*(i),H(i)\bigr)
	<
	\mu\bigl(x_H(i),H(i)\bigr).
\end{equation}

\medskip
\noindent\textbf{Claim.}
There exists a push $(x_H,x_H',x_H'')$ of $x_H$ towards $H$ 
such that
$\mu(x^*,x_H'')<\mu(x^*,x_H)$.

\smallskip
\noindent\emph{Case 1.} $x_H(i)\ge u(i)$.
Then~\eqref{eq:jnopt} implies $x_H(i)>x^*(i)$.
A first step $x'_H$ from $x_H$ at coordinate $i$ towards $x^*$
reduces the distance from $u(i)$, and hence from $H(i)$.
The second step $x_H''$ provided by the $2$-step axiom
completes the desired push.

\smallskip
\noindent\emph{Case 2.} $x_H(i)\le l(i)$.
Then similarly, $x_H(i)<x^*(i)$.
A first step at coordinate $i$ towards $x^*$ reduces the distance
from $l(i)$, and the argument proceeds exactly as in Case~1.

\smallskip
\noindent\emph{Case 3.} $l(i)< x_H(i)< u(i)$.
Then~\eqref{eq:jnopt} implies $x_H(i)\notin H(i)$.
Since $H$ is a sponge,
\[
x_H(i)-1,\; x_H(i)+1\in H(i).
\]
Any push at coordinate $i$ from $x_H$ towards $x^*$ guaranteed
by the $2$-step axiom yields the desired push
$(x_H,x_H',x_H'')$. 

\smallskip
In all cases, the chosen push satisfies the claim. 

By the claim, $x\rightarrow x_H''$ and
$\mu(x^*,x_H'')<\mu(x^*,x_H)$,
contradicting the minimality of $\mu(x^*,x_H)$.
Hence $x^*=x_H$, completing the proof.
\endproof

\begin{theorem}\label{thm:sponge_jump}
Let $J\subseteq\mathbb{Z}^n$ be a jump system and $H\subseteq\mathbb{Z}^n$ a sponge. Then $x\in J$ is not optimal if and only if there exists $x\in J$   with smaller distance  than $\mu(x,H)$ from $H_x$ or from at least one of the $2n$ escapes from $H_x$. An optimal $x\in J$ is   found by solving $(2n+1)n$ parity sponge problems.   
\end{theorem}

\prove 
The if part is easy: since $H_x, H_{F, \underline i}, H_{F, \overline i}\subseteq H$, for all $i\in [n]$, we have for each $x'\in J$: 
	
	\smallskip
	\centerline{$\displaystyle\mu(x',H)\le \min\big\{\mu(x',H_x),  \cup_{i=1}^n\{\mu(x',H_{x, \underline i}),  \mu(x',H_{x, \overline i})\} \big\},$ }
	
	\smallskip\noindent
	and therefore, if the right hand side is smaller than  $\mu(x,H)$,  then so is the left hand side, showing that $x$ is not optimal. 
	
	To prove the only if part, suppose that $x$ is not optimal,   and apply   Lemma~\ref{lem:push}: let 
	\begin{equation}\label{eq:jnotopt}
		x_H\in J,\,  x\rightarrow x_H,\, \mu(x_H, H)<\mu(x,H).
	\end{equation} 
	Let $\Fscr\subseteq J$ be the set of those points of $J$ that occur in the series of pushes starting with $x$ and finishing with $x_H$. 
	If for all $f\in\Fscr$ we have  $H_{f}=H_x$, we are done, for then $H_{x_H}=H_x$, and 
	\[\mu(x_H, H_x) =\mu(x_H,H_{x_H})= \mu(x_H,H)<\mu(x,H),\]
	where the first equality repeats $H_{x_H}=H_x$,  the second is  Fact~\ref{fact:environment}(i), and the inequality is (\ref{eq:jnotopt}).

	Otherwise, let $y\in \Fscr$, $H_{y}=H_x$ such that   the next push in the trajectory $x\rightarrow x_H$ is $(y,y',y'')$, where $y''$ is the  first point  of the trajectory with   $H_{y''}\ne H_x$.  Let this be a push at $i$ with second step at $j$. (Possibly $j=i$.) Then  $y'=y\pm e_i$. 
	
	 Suppose first that $y'(i)=y(i)-1$. Since $H_x=H_y$ we have  $H_{x, \underline j}=
	H_{y, \underline j}$. Applying Fact~\ref{fact:jelementary}~(iii) to $y''$ and the sponge $H_{y, \underline j}$ and then Fact~\ref{fact:jelementary}~(ii) successively along the  elementary changes through each $f\in\Fscr$ in the order they occur in the trajectory $x\rightarrow y$,  we obtain
	\[\mu(y'',H_{x, \underline j})=\mu(y'',H_{y, \underline j})< \mu(y,H)\le\ldots\le \mu(f,H)\le\ldots\le\mu(x,H),\]
	i.e. the distance of $y''$  from the escape $H_{x, \underline j}$ is strictly smaller than $\mu(x, H)$. If  $y'(i)=y(i)+1$, in an entirely symmetric way,  the distance of $y''$  from the escape $H_{x, \overline j}$ is smaller than $\mu(x, H)$.
	 
	 The final statement of the theorem  follows exactly as in  the proof of  Corollary~\ref{cor:complexity}.
\endproof
Corollary~\ref{cor:classicalvsparity} extends in the same manner, without the time bound.  

The following related problem was recalled by  Recski \cite{Sjump}, raised by applications in  \cite{recski,recskitakacs}.

\smallskip\noindent 
{\bf Problem}: Given a number $q\in\mathbb{N}$, a finite set $S$ with a partition and a  sponge associated with each class of the partition, and a matroid on $S$, is there an independent set of the matroid of size $q$, so that the number of elements in each partition class is in the sponge of the partition?

\smallskip
  This problem fits the sponge deficiency framework of jump systems, but Theorem~\ref{thm:sponge_jump} offered no immediate solution, as the classical case remained open. 
  
  Szabó~ \cite{jacint} resolved the parity deficiency problem for these jump systems via matroid parity in polynomial time, and using this,   Theorem~\ref{thm:sponge_jump} does indeed imply a solution of  Recski's problem in polynomial time. Furthermore,   this result is then applied in  \cite{jacint} to rigidity theory.

\medskip\noindent
{\bf \large Conclusion}: Tutte's article \cite{Tutte} is  a cornerstone result  in graph theory and a fertile source of combinatorial ideas,  opening the door to algebraic methods and complexity theory. 
We  traced a  line of development leading to increasingly general degree-constraints, and ultimately to a solution of the general factor problem in terms of jump systems, where new methods, notions, algorithms naturally occur.  This may have shown that Tutte's theorem served, and will hopefully also serve in the future, simultaneously as a source of motivation, a guiding goal, and a lasting companion. 
 
\medskip\noindent
{\bf \large Acknowledgment}: 
\small The author feels  fortunate to have been introduced to matching theory by András Frank and László Lovász as early as the 1970s and 1980s, leading to half a century of experience in the field. Many thanks are due to Erika Bérczi-Kovács, Louis Esperet,  Kazuo Murota,  and the three referees for their  helpful comments. Warm thanks are due to  Caroline Series and Stuart White for their suggestions far beyond the usual editorial work. Besides keeping the contact, Stuart's tireless, repeated proofreading and stylistic advice substantially improved the text.

\end{document}